\documentclass[pdflatex,sn-mathphys-num]{sn-jnl}
\usepackage[utf8]{inputenc}
\usepackage[english]{babel}
\usepackage[symbol]{footmisc}
\usepackage{amssymb,amsmath,amsthm,amsfonts,xcolor,enumerate,hyperref,comment,longtable,cleveref}

\usepackage{times}

\usepackage{pdflscape}
\usepackage{ulem}
\usepackage[mathcal]{euscript}
\usepackage{tikz}
\usepackage{array}
\usepackage{hyperref}
\usepackage{cancel}
\usepackage{stmaryrd}

\usetikzlibrary{arrows}

\theoremstyle{plain}
\newtheorem{Th}{Theorem}

\newtheorem{defn}[Th]{Definition}

\usepackage{stmaryrd}
\usepackage{xcolor}

\begin{document}

\title[$n$-dimensional totally graded filiform Lie algebras]{Local and 2 local $\frac{1}{2}$-derivation of $n$-dimensional totally graded filiform Lie algebras}

\author*[1,2]{\fnm{Farkhodzhon} \sur{Arzikulov}}\email{arzikulovfn@gmail.com}

\author[1,3]{\fnm{Mirzobek} \sur{Shodiev}}\email{mirzobekshodiyev7@gmail.com}
\equalcont{These authors contributed equally to this work.}

\affil*[1]{\orgdiv{V.I. Romanovskiy Institute of Mathematics}, \orgname{Uzbekistan Academy of Sciences}, \orgaddress{\street{Olmazor district, University street 9}, \city{Tashkent}, \postcode{100174}, \country{Uzbekistan}}}

\affil[2]{\orgdiv{Mathematics}, \orgname{Andijan State University}, \orgaddress{\street{University 129}, \city{Andijan}, \postcode{170100}, \country{Uzbekistan}}}

\affil[3]{\orgdiv{Mathematics}, \orgname{Bukhara State University}, \orgaddress{\street{Mirzo Iqbol street 11}, \city{Bukhara}, \postcode{160107}, \country{Uzbekistan}}}

\abstract  {This article provides a complete algebraic description of $\frac{1}{2}$-derivations, local $\frac{1}{2}$-derivations, and 2-local $\frac{1}{2}$-derivations on $n$-dimensional totally graded complex filiform Lie algebras of maximum length. Based on the foundational classification framework established by Janez Bernik (2020), we systematically determine the vector spaces of  $\frac{1}{2}$-derivations for the six infinite structural sequences ($m_0(n)$, $m_2(n)$, $W^+(n)$, $m_{0,1}(n)$, $m_{0,2}(n)$, $m_{0,3}(n)$) and the five exceptional one-parameter families ($g_{7,\alpha}$ through $g_{11,\alpha}$). By analyzing the point-wise local evaluation equations via parametric matrix systems, we establish the structural linearity and rigidity of local $\frac{1}{2}$-derivations. In contrast, we demonstrate that the independent parameters residing in the boundary rows of the  $\frac{1}{2}$-derivation matrices provide sufficient degrees of freedom to bypass linearity constraints. Exploiting these boundary configurations, we explicitly construct pure non-linear and non-additive 2-local $\frac{1}{2}$-derivations leveraging the homogeneous function of degree one, $f(z_1, z_2) = z_1^3 / (z_1^2 + z_2^2)$, thereby defining the exact boundary where local rigidity fails.

}



\maketitle

\section{Introduction}\label{sec1}
The structural analysis of nilpotent Lie algebras is a central theme in modern non-associative algebra. Among these, filiform Lie algebras possess a maximal nilpotency index. When equipped with a connected integer grading where each homogeneous component is one-dimensional, they are called totally graded filiform Lie algebras of maximum length. Recently, Janez Bernik \cite{Bernik2020} refined the classification of these structures, partitioning them into six infinite sequences ($m_0(n)$, $m_2(n)$, $W^+(n)$, $m_{0,1}(n)$, $m_{0,2}(n)$, $m_{0,3}(n)$) and five exceptional one-parameter families ($g_{7,\alpha}$ through $g_{11,\alpha}$).

Concurrently, generalized and localized operator spaces on Lie algebras have attracted significant interest. Originally introduced by Filippov \cite{fil1}, a linear map $\Delta: \mathfrak{g} \to \mathfrak{g}$ is a $\frac{1}{2}$-derivation if it satisfies the identity:
\[
\Delta([x, y]) = \frac{1}{2}\bigl([\Delta(x), y] + [x, \Delta(y)]\bigr)
\]
for all $x, y \in \mathfrak{g}$. Interest in these maps intensified after Ferreira, Kaygorodov, and Lopatkin \cite{FKL} proved their duality with transposed Poisson algebras. Consequently, $\frac{1}{2}$-derivations have been classified across various infinite-dimensional and simple frameworks \cite{Bai, FKL}.

Parallel to linear derivations, localized algebraic operators have evolved into a prominent paradigm. Based on the foundational works of Kadison \cite{Kadison} and Larson-Sourour \cite{Larson}, a local derivation coincides with a  derivation at each single point. Extending this, Šemrl \cite{Semrl} introduced 2-local derivations, which match  derivations on any pair of elements. While local and 2-local standard derivations or automorphisms have been widely studied for nilpotent configurations \cite{Ayupov2016, Kobiljon}, including recent works on totally graded filiform structures \cite{2}, their behavior under the $\frac{1}{2}$-derivation constraint remains completely unexplored for maximal length filiform systems.

The primary objective of this paper is to bridge this gap by providing a comprehensive algebraic description of , local, and 2-local $\frac{1}{2}$-derivations for all eleven structures in Bernik's classification \cite{Bernik2020}. We explicitly solve the parametric matrix-vector coordination systems derived from point-wise mappings to determine where local linearity holds and how internal parameters affect these spaces.

The paper is organized as follows:
In Section 2, we recall the essential definitions and algebraic preliminaries for maximum length filiform Lie algebras.
In Section 3, we present the explicit determination and full classification proofs for the vector spaces of  $\frac{1}{2}$-derivations.
In Section 4, we study local $\frac{1}{2}$-derivations, establishing filtration protocols over coordinate subspaces to prove their strict containment within  structures.
Finally, Section 5 treats the 2-local $\frac{1}{2}$-derivation problem with full mathematical proofs. By exploiting independent parameters in the boundary rows of the  matrices, we construct non-linear, non-additive 2-local transformations using the homogeneous function:
\[
f(z_1, z_2) = \frac{z_1^3}{z_1^2 + z_2^2},
\]
defining the exact boundary where local linearity and rigidity fail.

\section{Preliminaries}\label{sec2}

In this section we give necessary definitions and preliminary results for Lie algebras.

\begin{defn}
A vector space $L$ over a field $\mathbb{F}$ equipped with a bilinear bracket $[\cdot,\cdot]: L\times L\to L$ is called a \textbf{Lie algebra} if for all $x,y,z\in L$ the following identities hold:
\begin{itemize}
\item \textbf{Skew-symmetry:} $[x,x]=0$ (equivalently $[x,y]=-[y,x]$),
\item \textbf{Jacobi identity:} $[x,[y,z]]+[y,[z,x]]+[z,[x,y]]=0$.
\end{itemize}
\end{defn}

For a Lie algebra $L$, the Jacobi identity can be written as
\[
[x,[y,z]] = [[x,y],z] + [y,[x,z]].
\]

The \textbf{lower central series} of a Lie algebra $L$ is defined recursively by
\[
L^1 = L,\qquad L^{k+1}=[L^k,L],\quad k\ge 1.
\]

\begin{defn}
A Lie algebra $L$ is called \textbf{nilpotent} if there exists $k\in\mathbb{N}$ such that $L^k=\{0\}$. The smallest such $k$ is called the \textbf{index of nilpotency} of $L$.
\end{defn}
For an $n$-dimensional nilpotent Lie algebra $L$, the nilpotency index (the smallest $k$ such that $L^k = \{0\}$) satisfies $k \le n$. The maximal index $n$ is attained precisely by the \emph{filiform} Lie algebras. Indeed, for the standard filiform algebra with basis $\{e_1,\dots,e_n\}$ and non-zero brackets $[e_1,e_i]=e_{i+1}$ ($2\le i\le n-1$), the lower central series is
\[
L^2 = \langle e_3,\dots,e_n\rangle,\; L^3 = \langle e_4,\dots,e_n\rangle,\; \dots,\; L^{n-1} = \langle e_n\rangle,\; L^n = 0,
\]
so the nilpotency index equals $n$. For the $n$-dimensional abelian Lie algebra we have $L^2=0$, i.e., index $2$. Consequently, the upper bound is $n$, not $n+1$. A nilpotent Lie algebra is called \emph{quasi-filiform} if its nilpotency index is $n-1$ (i.e., $L^{n-1}=0$ but $L^{n-2}\neq 0$).

\begin{defn}
An $n$-dimensional nilpotent Lie algebra is called \textbf{filiform} if its index of nilpotency is $n$ (i.e., $L^n=0$ but $L^{n-1}\neq0$). It is called \textbf{quasi-filiform} if its index of nilpotency is $n-1$ (i.e., $L^{n-1}=0$ but $L^{n-2}\neq0$).
\end{defn}

A Lie algebra $L$ is said to be $\mathbb{Z}$-graded if $L=\bigoplus_{i\in\mathbb{Z}} V_i$ with $[V_i,V_j]\subseteq V_{i+j}$ for all $i,j$, and only finitely many $V_i$ are non-zero. A gradation is called \textbf{connected} if $L=V_{k_1}\oplus\cdots\oplus V_{k_t}$ with $V_{k_i}\neq\{0\}$ for every $i$ and $k_1<k_2<\cdots<k_t$.

\begin{defn}\label{halfderiv}
  Let $(\mathfrak{L}, [-,-])$ be an algebra with a multiplication $[-,-],$ and $\varphi$ be
a linear map. Then $\varphi$ is a $\frac12$-\textbf{derivation} if it satisfies:
\begin{equation}\label{halfderiv1}
\varphi([x, y]) = \frac12([\varphi(x), y] + [x, \varphi(y)]).
\end{equation}
\end{defn}

Note that $\frac{1}{2}$-derivations are a particular case of $\delta$-derivations introduced by Filippov in \cite{fil1}
.  It is easy to see from Definition \ref{halfderiv} that $[\mathfrak{L},\mathfrak{L}]$ and $Ann(\mathfrak{L})$ are invariant under any $\frac 12$-derivation of $\mathfrak{L}$. 
\begin{defn}
Let $\mathfrak{g}$ be an $n$-dimensional nilpotent Lie algebra over a field $\mathbb{F}$.
The algebra $\mathfrak{g}$ is called a \textbf{totally graded filiform Lie algebra}
(or a filiform Lie algebra of maximum length) if it satisfies the following two conditions concurrently:
\begin{enumerate}
    \item \textbf{Filiform Property:} Its nilpotency class is maximal, meaning that for the lower central series defined by $C^1\mathfrak{g} = \mathfrak{g}$ and $C^{i+1}\mathfrak{g} = [\mathfrak{g}, C^i\mathfrak{g}]$ for $i \ge 1$, the dimensions satisfy:
    \[ \dim C^i\mathfrak{g} = n - i \quad \text{for } 2 \le i \le n \]
    Consequently, $C^{n-1}\mathfrak{g} \neq \{0\}$ and $C^n\mathfrak{g} = \{0\}$.

    \item \textbf{Total Grading Property:} It admits an $\mathbb{N}$-grading as a direct sum of 1-dimensional homogeneous vector subspaces:
    \[ \mathfrak{g} = \mathfrak{g}_1 \oplus \mathfrak{g}_2 \oplus \dots \oplus \mathfrak{g}_n \]
    such that $[\mathfrak{g}_i, \mathfrak{g}_j] \subseteq \mathfrak{g}_{i+j}$ for all $i, j \in \{1, \dots, n\}$, where $\dim \mathfrak{g}_i = 1$ for each $1 \le i \le n$, and $\mathfrak{g}_i = \{0\}$ for all $i > n$.
\end{enumerate}

\end{defn}

\begin{Th}[\cite{Bernik2020}]
Let \( g = \bigoplus_{i=1}^n g_i \) be a totally graded filiform Lie algebra. Then \( g \) is isomorphic to a Lie algebra from the following list:
\end{Th}

\begin{enumerate}
    \item Filiform Lie algebras belonging to the six sequences \( m_0(n) \), \( m_2(n) \), \( W^+(n) \), \( m_{0,1}(2k+1) \), \( m_{0,2}(2k+2) \), \( m_{0,3}(2k+3) \), defined by the homogeneous basis \( \{e_1, \ldots, e_n\} \) and commutating relations given in Table 1;
    \item Filiform Lie algebras of the 5 one-parameter families \( g_{n,\alpha} \) of dimensions \( n = 7, \ldots, 11 \) respectively, defined by their homogeneous bases and commutation relations given in Table 2.
\end{enumerate}

\begin{table}[h]
\centering
\caption{Six infinite sequences of totally graded filiform Lie algebras}
\label{tab:filiformn}
\begin{tabular}{|c|c|>{\raggedright\arraybackslash}p{10cm}|}
\hline
\textbf{Algebra} & \textbf{Dimension} & \textbf{Commutation relations} \\
\hline
\( m_0(n) \) & \( n \ge 3 \) & \([e_1, e_i] = e_{i+1},\ i = 2, \ldots, n-1\) \\
\hline
\( m_2(n) \) & \( n \ge 5 \) &
\([e_1, e_i] = e_{i+1},\ i = 2, \ldots, n-1\) \\
& & \([e_2, e_i] = e_{i+2},\ i = 3, \ldots, n-2\) \\
\hline
\( W^+(n) \) & \( n \ge 12 \) &
\([e_i, e_j] =
\begin{cases}
(j-i)e_{i+j}, & i+j \le n; \\
0, & i+j > n;
\end{cases} \) \\
\hline
\( m_{0,1}(2k+1) \) & \( n = 2k+1,\ k \ge 3 \) &
\([e_1, e_i] = e_{i+1},\ i = 2, \ldots, 2k\) \\
& & \([e_l, e_{2k-l+1}] = (-1)^{l+1}e_{2k+1},\ l = 2, \ldots, k\) \\
\hline
\( m_{0,2}(2k+2) \) & \( n = 2k+2,\ k \ge 3 \) &
\([e_1, e_i] = e_{i+1},\ i = 2, \ldots, 2k+1\) \\
& & \([e_l, e_{2k-l+1}] = (-1)^{l+1}e_{2k+1},\ l = 2, \ldots, k\) \\
& & \([e_j, e_{2k-j+2}] = (-1)^{j+1}(k-j+1)e_{2k+2},\ j = 2, \ldots, k\) \\
\hline
\( m_{0,3}(2k+3) \) & \( n  = 2k+3,\ k \ge 3 \) &
\([e_1, e_i] = e_{i+1},\ i = 2, \ldots, 2k+2\) \\
& & \([e_l, e_{2k-l+1}] = (-1)^{l+1}e_{2k+1},\ l = 2, \ldots, k\) \\
& & \([e_j, e_{2k-j+2}] = (-1)^{j+1}(k-j+1)e_{2k+2},\ j = 2, \ldots, k\) \\
& & \([e_m, e_{2k-m+3}] = (-1)^m\left((m-2)k - \frac{(m-2)(m-1)}{2}\right)e_{2k+3},\ m = 3, \ldots, k+1\) \\
\hline
\end{tabular}
\end{table}

\begin{table}[htbp]
\centering
\caption{Five one-parameter families of totally graded filiform Lie algebras.}
\label{tab:filiform}
\begin{tabular}{|c|>{\raggedright\arraybackslash}p{0.7\linewidth}|}
\hline
\textbf{Algebra} & \textbf{Commutation relations} \\
\hline
$\mathfrak{g}_{7,\alpha}$ &
$\begin{array}{@{}l@{}}
[e_1, e_j] = e_{j+1},\quad 2 \leq j \leq 6
[e_2, e_3] = (2 + \alpha)e_5, \\[2pt] [e_2, e_4] = (2 + \alpha)e_6, \quad
[e_2, e_5] = (1 + \alpha)e_7,\quad [e_3, e_4] = e_7
\end{array}$ \\
\hline
$\mathfrak{g}_{8,\alpha}$ &
$\begin{array}{@{}l@{}}
\text{relations of } \mathfrak{g}_{7,\alpha} \text{ and} \quad
[e_1, e_7] = e_8,\quad [e_2, e_6] = \alpha e_8,\quad [e_3, e_5] = e_8,
\end{array}$ \\
\hline
$\mathfrak{g}_{9,\alpha}$ &
$\begin{array}{@{}l@{}}
\text{relations of } \mathfrak{g}_{8,\alpha} \text{ and} \quad
(\alpha \neq -\tfrac{5}{2}) \quad
[e_1, e_8] = e_9,\\[4pt] [e_2, e_7] = \dfrac{2\alpha^2 + 3\alpha - 2}{2\alpha + 5} e_9, \quad
[e_3, e_6] = \dfrac{2\alpha + 2}{2\alpha + 5} e_9,\quad [e_4, e_5] = \dfrac{3}{2\alpha + 5} e_9,
\end{array}$ \\
\hline
$\mathfrak{g}_{10,\alpha}$ &
$\begin{array}{@{}l@{}}
\text{relations of } \mathfrak{g}_{9,\alpha} \text{ and} \quad
(\alpha \neq -\tfrac{5}{2})
[e_1, e_9] = e_{10}, \\[2pt] [e_2, e_8] = \dfrac{2\alpha^2 + \alpha - 1}{2\alpha + 5} e_{10}, \quad
[e_3, e_7] = \dfrac{2\alpha - 1}{2\alpha + 5} e_{10},\quad [e_4, e_6] = \dfrac{3}{2\alpha + 5} e_{10},
\end{array}$ \\
\hline
$\mathfrak{g}_{11,\alpha}$ &
$\begin{array}{@{}l@{}}
\text{relations of } \mathfrak{g}_{10,\alpha} \text{ and} \\
(\alpha \neq -\tfrac{5}{2},\,-1,\,-3) \\[2pt]
[e_1, e_{10}] = e_{11},\quad [e_2, e_9] = \dfrac{2\alpha^3 + 2\alpha^2 + 3}{2(\alpha^2 + 4\alpha + 3)} e_{11}, \\[6pt]
[e_3, e_8] = \dfrac{4\alpha^3 + 8\alpha^2 - 8\alpha - 21}{2(\alpha^2 + 4\alpha + 3)(2\alpha + 5)} e_{11}, \\[6pt]
[e_4, e_7] = \dfrac{3(2\alpha^2 + 4\alpha + 5)}{2(\alpha^2 + 4\alpha + 3)(2\alpha + 5)} e_{11}, \\[6pt]
[e_5, e_6] = \dfrac{3(4\alpha + 1)}{2(\alpha^2 + 4\alpha + 3)(2\alpha + 5)} e_{11}
\end{array}$ \\
\hline
\end{tabular}
\end{table}

\newpage

\section{\texorpdfstring{$\frac{1}{2}$-derivation of a totally graded filiform Lie algebras}{derivations of a filiform Lie algebra}.}\label{sec3}

\subsection{\texorpdfstring{$\frac{1}{2}$-derivation of the algebra $m_0(n)$}{1/2 derivations of the algebra m0(n)}.}\label{subsec2}

In this section, we investigate and compute $\frac{1}{2}$-derivation of quasi-filiform Lie algebras, starting with the following theorem.

\begin{Th}\label{THM: halfder m0n} Any $\frac{1}{2}$-derivation of the algebra $m_0(n)$($n\ge 3$)
has the form

\begin{equation}\label{halfder: m0n}
\begin{cases}
\varphi(e_1)=\sum\limits_{i=1}^{n}\alpha_ie_i, \\ \varphi(e_2)=\sum\limits_{i=2}^{n}\beta_ie_i,, \\  
\varphi(e_k)=\frac{2^{k-2}-1}{2^{k-2}}\alpha_1 e_k+\frac{1}{2^{k-2}}\sum\limits_{i=2}^{n-k+2}\beta_{i}e_{i+k-2}, \  3\leq k\leq n-1,\\
\varphi(e_n)=\frac{(2^{n-2}-1)\alpha_1+\beta_2}{2^{n -2}} e_n. 
\end{cases}
\end{equation}

\end{Th}

\begin{proof}
We exploit it to determine the $\frac{1}{2}$-derivation.
$$\varphi(e_1)=\sum\limits_{i=1}^{n}\alpha_ie_i,\quad \varphi(e_2)=\sum\limits_{i=1}^{n}\beta_ie_i.$$

Now consider the condition of $\frac 1 2$-derivation for the elements $e_1$ and $e_2:$
\begin{center}
$\varphi(e_3)=\varphi([e_1,e_2])=\frac{1}{2}([\varphi(e_1),e_2]+[e_1,\varphi(e_2)])=\frac{1}{2}([\sum\limits_{i=1}^{n}\alpha_ie_i,e_2]+[e_1,\sum\limits_{i=1}^{n}\beta_ie_i])
=\frac{1}{2}(\alpha_1 e_3+\sum\limits_{i=2}^{n-1}\beta_{i}e_{i+1}).$
\end{center}
According to the structure of \(m_0(n)\), we have \([e_2, e_3] = 0\). Applying the definition of a \(\frac{1}{2}\)-derivation yields \(\varphi ([e_{2},e_{3}])=\frac{1}{2}([\varphi (e_{2}),e_{3}]+[e_{2},\varphi (e_{3})])=\frac{1}{2}[\varphi (e_{2}),e_{3}]=b_{1}e_{4}=0,\) which directly implies \(b_1=0\). Based on this initial relation, we proceed by induction to evaluate the remaining cases

We prove the following equality for $3\leq k\leq n-1$ by induction:
\begin{center}
$\varphi(e_k)=\frac{2^{k-2}-1}{2^{k-2}}\alpha_1 e_k+\frac{1}{2^{k-2}}\sum\limits_{i=2}^{n-k+2}\beta_{i}e_{i+k-2}$
\end{center}

The base case \(k=3\) is immediately satisfied by the relation above. Now, assuming the hypothesis holds for \(k\), we extend the argument to \(k+1\). By evaluating the \(\frac{1}{2}\)-derivation property on the elements \(e_{1}\) and \(e_{k}\), we obtain

\begin{center}
$\varphi(e_{k+1})=\varphi([e_1,e_k])=\frac{1}{2}([\varphi(e_1),e_k]+[e_1,\varphi(e_k)])=\frac{1}{2}\Big([\sum\limits_{i=1}^{n}\alpha_ie_i,e_k]+[e_1,\frac{2^{k-2}-1}{2^{k-2}}\alpha_1 e_k+\frac{1}{2^{k-2}}\sum\limits_{i=2}^{n-k+2}\beta_{i}e_{i+k-2}]\Big)$
\end{center}
\begin{center}
$=\frac{1}{2}\Big(\alpha_{1}e_{k+1}+\frac{1}{2^{k-2}}(2^{k-2}-1)\alpha_{1}e_{k+1}+\frac{1}{2^{k-2}}\sum\limits_{i=2}^{n-k+1}\beta_{i}e_{i+k-1}\Big)=\frac{2^{k-1}-1}{2^{k-1}}\alpha_1 e_{k+1}+\frac{1}{2^{k-1}}\sum\limits_{i=2}^{n-k+1}\beta_{i}e_{i+k-1}.$
\end{center}

Now, consider the condition of $\frac 1 2$-derivation for the elements $e_1, e_{n-1}:$
\begin{center}
$\varphi(e_{n})=\varphi([e_1,e_{n-1}])=\frac{1}{2}([\varphi(e_1),e_{n-1}]+[e_1,\varphi(e_{n-1})])$
\end{center}
\begin{center}
$=\frac{1}{2}\Big([\sum\limits_{i=1}^{n}\alpha_ie_i,e_{n-1}]+[e_1,\frac{2^{n-3}-1}{2^{n-3}}\alpha_1 e_{n-1}+\frac{1}{2^{n-3}}(b_2e_{n-1}+b_3e_n)  ]\Big)
=\frac{(2^{k-2}-1)\alpha_1+\beta_2}{2^{k-2}} e_n$
\end{center}
Thus, we obtain that
\begin{equation}\label{fien1}
\varphi(e_{n})=\frac{(2^{k-2}-1)\alpha_1+\beta_2}{2^{k-2}} e_n.
\end{equation}

This completes the proof of the theorem. \end{proof}

\subsection{\texorpdfstring{$\frac{1}{2}$-derivation of the algebra $m_2(n)$}{1/2 derivations of the algebra m2(n)}.}\label{subsec3}

\begin{Th}\label{THM: halfder m2n} Any $\frac{1}{2}$-derivation of the algebra $m_2(n)$ ($n\ge 5$)
has the form

\begin{equation}\label{halfder: m2n}
\begin{cases}
\varphi(e_1)=\sum\limits_{i=1}^{n}\alpha_ie_i, \\ \varphi(e_2)=\alpha_1 e_2+ 6\alpha_2e_3+ \sum\limits_{i=3}^{n-4}\alpha_ie_{i+1}+\sum\limits_{i=n-2}^{n}\beta_ie_{i}, \\
\varphi(e_3)= \alpha_1e_3+3\alpha_2e_4+\frac{1}{2}\bigl( (\beta_{n-2}- \alpha_{n-3})e_{n-1} + (\beta_{n-1}- \alpha_{n-2})e_{n} \bigl) ,            \\
\varphi(e_4)= \alpha_1e_4+2\alpha_2e_5+\frac{1}{4}\bigl( (\beta_{n-2}- \alpha_{n-3})e_{n} \bigl), \\
\varphi(e_k)=\alpha_1 e_k+\frac{2^{k-4}+1}{2^{k-4}} \alpha_2 e_{k+1}, \  5\leq k\leq n-1,\\
\varphi(e_n)=\alpha_1e_n. 
\end{cases}
\end{equation}

\end{Th}

\begin{proof}
Since \(m_2(n)\) has two multiply, we exploit it to determine the $\frac{1}{2}$-derivation.
$$\varphi(e_1)=\sum\limits_{i=1}^{n}\alpha_ie_i,\quad \varphi(e_2)=\sum\limits_{i=1}^{n}\beta_ie_i.$$

Now consider the condition of $\frac 1 2$-derivation for the elements $e_1$ and $e_i:$
\begin{center}
$\varphi(e_3)=\varphi([e_1,e_2])=\frac{1}{2}([\varphi(e_1),e_2]+[e_1,\varphi(e_2)]) =\frac{1}{2}([\sum\limits_{i=1}^{n}\alpha_ie_i,e_2]+[e_1,\sum\limits_{i=1}^{n}\beta_ie_i])
=\frac{1}{2}(\alpha_1 e_3-\sum\limits_{i=3}^{n-2}\alpha_{i}e_{i+2}+\sum\limits_{i=2}^{n-1}\beta_{i}e_{i+1}).$
\end{center}
\begin{center}
$\varphi(e_4)=\varphi([e_1,e_3])=\frac{1}{2}([\varphi(e_1),e_3]+[e_1,\varphi(e_3)]) =\frac{1}{2}([\sum\limits_{i=1}^{n}\alpha_ie_i,e_3]+[e_1,\frac{1}{2}(\alpha_1 e_3-\sum\limits_{i=3}^{n-2}\alpha_{i}e_{i+2}+\sum\limits_{i=2}^{n-1}\beta_{i}e_{i+1}])
$
\end{center}
\begin{center}
    $=\frac{1}{2} \bigl(\alpha_1 e_4+\alpha_2e_5+\frac{1}{2}(\alpha_1 e_4-\sum\limits_{i=3}^{n-3}\alpha_{i}e_{i+3}+\sum\limits_{i=2}^{n-2}\beta_{i}e_{i+2}) \bigl).$
\end{center}
\begin{equation}\label{varpi(e_5)_1 m2n}
\begin{aligned}
    \varphi(e_5)=\varphi ([e_1,e_4])=\frac{1}{2}([\varphi(e_1),e_4]+[e_1,\varphi(e_4)])= \\
    =\frac{1}{2} \Bigl([\sum\limits_{i=1}^{n}\alpha_ie_i,e_4]+[e_1,\frac{1}{2} \bigl(\alpha_1 e_4+\alpha_2e_5+\frac{1}{2}(\alpha_1 e_4-\sum\limits_{i=3}^{n-3}\alpha_{i}e_{i+3}+\sum\limits_{i=2}^{n-2}\beta_{i}e_{i+2}) \bigl)]\Bigl) \\ = \frac{7}{8}\alpha_1e_5+\frac{3}{4}\alpha_2e_6+\frac{1}{8}\Bigl( \sum\limits_{i=2}^{n-3}\beta_{i}e_{i+3} - \sum\limits_{i=3}^{n-4}\alpha_{i}e_{i+4} \Bigl)
\end{aligned}
\end{equation}

According to the structure of \(m_2(n)\), we have \([e_2, e_3] = e_5\). Applying the definition of a \(\frac{1}{2}\)-derivation yields 

\begin{equation}\label{varpi(e_5)_2 m2n}
\begin{aligned}
    \varphi(e_5)=\varphi ([e_{2},e_{3}])=\frac{1}{2}([\varphi (e_{2}),e_{3}]+[e_{2},\varphi (e_{3})])= \frac{1}{2} \Bigl([\sum\limits_{i=1}^{n}\beta_ie_i,e_3]+[e_2,\frac{1}{2}(\alpha_1 e_3-\sum\limits_{i=3}^{n-2}\alpha_{i}e_{i+2}+\sum\limits_{i=2}^{n-1}\beta_{i}e_{i+1})]\Bigl) \\ = \frac{1}{2}\beta_1e_4+(\frac{1}{2}\beta_2+\frac{1}{4}\alpha_1) e_5+\frac{1}{4}\Bigl( \sum\limits_{i=2}^{n-3}\beta_{i}e_{i+3} - \sum\limits_{i=3}^{n-4}\alpha_{i}e_{i+4} \Bigl)
\end{aligned}
\end{equation}

By equating the coefficients of the basis elements in the expressions for \(\varphi(e_5)\) in equations \eqref{varpi(e_5)_1 m2n} and \eqref{varpi(e_5)_2 m2n}, we obtain.

$$\begin{cases}
   e_4: \quad \frac{1}{2}\beta_1=0   \\
   e_5: \quad  \frac{7}{8}\alpha_1+\frac{1}{8}\beta_2=\frac{3}{4}\beta_2+\frac{1}{4}\alpha_1   \\ 
    e_6: \quad \frac{3}{4}\alpha_2+\frac{1}{8}\beta_3=\frac{1}{4}\beta_3  \\
    \frac{1}{8}\Bigl( \sum\limits_{i=4}^{n-3}\beta_{i}e_{i+3} - \sum\limits_{i=3}^{n-4}\alpha_{i}e_{i+4}
    \Bigl)= \frac{1}{4}\Bigl( \sum\limits_{i=4}^{n-3}\beta_{i}e_{i+3} - \sum\limits_{i=3}^{n-4}\alpha_{i}e_{i+4} \Bigl) 
\end{cases}$$
so we take $\beta_1=0, \quad \beta_2= a_1, \quad  \beta_3=6 \alpha_2  $ and so last cases: $ \sum\limits_{i=4}^{n-3}\beta_{i}e_{i+3} - \sum\limits_{i=3}^{n-4}\alpha_{i}e_{i+4}=0 \Longrightarrow \beta_i=\alpha_{i-1} \quad where \quad 4 \le i \le n-3$

We prove the following equality for $5\leq k\leq n-1$ by induction:
$$\varphi(e_k)=\alpha_1 e_k+\frac{2^{k-4}+1}{2^{k-4}} \alpha_2 e_{k+1}$$

if $k=5$ 
$$\varphi(e_5)=\varphi ([e_1,e_4])=\frac{1}{2}([\varphi(e_1),e_4]+[e_1,\varphi(e_4)])=\alpha_1 e_5+\frac{3}{2} \alpha_2 e_{6}$$

The base case \(k=5\) is immediately satisfied by the relation above. Now, assuming the hypothesis holds for \(k\), we extend the argument to \(k+1\). By evaluating the \(\frac{1}{2}\)-derivation property on the elements \(e_{1}\) and \(e_{k}\), we obtain

\begin{center}
$\varphi(e_{k+1})=\varphi([e_1,e_k])=\frac{1}{2}([\varphi(e_1),e_k]+[e_1,\varphi(e_k)])=\frac{1}{2}\Big([\sum\limits_{i=1}^{n}\alpha_ie_i,e_k]+[e_1,\alpha_1 e_k+\frac{2^{k-4}+1}{2^{k-4}} \alpha_2 e_{k+1}]\Big)$
\end{center}
\begin{center}
$=\frac{1}{2}\Big( \alpha_1e_{k+1}+\alpha_2e_{k+2}+\alpha_1 e_{k+1}+\frac{2^{k-4}+1}{2^{k-4}} \alpha_2 e_{k+2}  \Big)= \alpha_1 e_{k+1}+\frac{2^{k-3}+1}{2^{k-3}} \alpha_2 e_{k+2}$
\end{center}

Now, consider the condition of $\frac 1 2$-derivation for the elements $e_1, e_{n-1}:$
\begin{center}
$\varphi(e_{n})=\varphi([e_1,e_{n-1}])=\frac{1}{2}([\varphi(e_1),e_{n-1}]+[e_1,\varphi(e_{n-1})])=\frac{1}{2}\Big([\sum\limits_{i=1}^{n}\alpha_ie_i,e_{n-1}]+[e_1,\alpha_1 e_k+\frac{2^{k-4}+1}{2^{k-4}} \alpha_2 e_{k+1}  ]\Big)$
\end{center}
\begin{center}
$= \frac{1}{2}\Big( \alpha_1 e_{n} +\alpha_1 e_{n} \Big)= \alpha_1 e_{n}
$
\end{center}
Thus, we obtain that
\begin{equation}\label{varfi en m2n}
\varphi(e_{n})=\alpha_1 e_{n}
\end{equation}

\end{proof}

\subsection{\texorpdfstring{$\frac{1}{2}$-derivation of the algebra $W^+(n)$}{1/2 derivations of the algebra w+(n)}.}\label{subsec4}

\begin{Th}\label{THM: halfder w+n} Any $\frac{1}{2}$-derivation of the algebra $W^+(n)$ ($n\ge 12$)
has the form

\begin{equation}\label{halfder: w+n}
\begin{cases}
\varphi(e_1)=\sum\limits_{i=1}^{n}\alpha_ie_i, \\ \varphi(e_2)=\sum\limits_{i=1}^{n-4}\alpha_ie_{i+1}+\sum\limits_{i=n-2}^{n}\beta_ie_{i}, \\
\varphi(e_3)= \sum\limits_{i=1}^{n-4}\alpha_ie_{i+2} + \frac{1}{2} \Bigl(   (n-3)\beta_{n-2} -(n-7)\alpha_{n-3} \Bigl)e_{n-1}+  \frac{1}{2} \Bigl(   (n-2)\beta_{n-1} -(n-6)\alpha_{n-2} \Bigl)e_{n}     \\
\varphi(e_4)= \sum\limits_{i=1}^{n-4}\alpha_ie_{i+3} + \frac{1}{8} \Bigl(   (n^2-5n+6)\beta_{n-2} -(n^2-5n-10)\alpha_{n-3} \Bigl)e_{n}, \\
\varphi(e_k)=\sum\limits_{i=1}^{n-k+1}\alpha_ie_{i+k-1}, \  5\leq k\leq n-1,\\
\varphi(e_n)=\alpha_1e_n. 
\end{cases}
\end{equation}

\end{Th}

\begin{proof}

$$\varphi(e_1)=\sum\limits_{i=1}^{n}\alpha_ie_i,\quad \varphi(e_2)=\sum\limits_{i=1}^{n}\beta_ie_i.$$

Now consider the condition of $\frac 1 2$-derivation for the elements $e_1$ and $e_i:$
\begin{center}
$\varphi(e_3)=\varphi([e_1,e_2])=\frac{1}{2}([\varphi(e_1),e_2]+[e_1,\varphi(e_2)]) $
\end{center}
\begin{center}
$=\frac{1}{2}([\sum\limits_{i=1}^{n}\alpha_ie_i,e_2]+[e_1,\sum\limits_{i=1}^{n}\beta_ie_i])
=\frac{1}{2}(\alpha_1 e_3-\sum\limits_{i=3}^{n-2} (i-2)\alpha_{i}e_{i+2}+\sum\limits_{i=2}^{n-1}(i-1) \beta_{i}e_{i+1}).$
\end{center}
\begin{center}
$\varphi(e_4)=\varphi( \frac{1}{2} [e_1,e_3])=\frac{1}{4}([\varphi(e_1),e_3]+[e_1,\varphi(e_3)]) $
\end{center}
\begin{center}
$= \frac{1}{4} \bigl([\sum\limits_{i=1}^{n}\alpha_ie_i,e_3]+[e_1,\frac{1}{2}(\alpha_1 e_3-\sum\limits_{i=3}^{n-2} (i-2)\alpha_{i}e_{i+2}+\sum\limits_{i=2}^{n-1}(i-1) \beta_{i}e_{i+1})]\bigl)
$
\end{center}
\begin{center}
    $= \frac{1}{4} \bigl(2\alpha_1 e_4+\alpha_2e_5-2 \alpha_3e_6+\frac{1}{2}(-\sum\limits_{i=4}^{n-3} (i^2+i-8)\alpha_{i}e_{i+3}+\sum\limits_{i=2}^{n-2}i(i-2)\beta_{i}e_{i+2}) \bigl).$
\end{center}
\begin{equation}\label{varpi(e_5)_1 w+n}
\begin{aligned}
    \varphi(e_5)=\varphi (
    \frac{1}{3}[e_1,e_4])=\frac{1}{6}([\varphi(e_1),e_4]+[e_1,\varphi(e_4)])= \\
    =\frac{1}{6} \Bigl([\sum\limits_{i=1}^{n}\alpha_ie_i,e_4]+[e_1,  \frac{1}{4} \Bigl(  2\alpha_1 e_4+\alpha_2e_5-2 \alpha_3e_6+\frac{1}{2}(-\sum\limits_{i=4}^{n-3} (i^2+i-8)\alpha_{i}e_{i+3}+\sum\limits_{i=2}^{n-2}i(i-2)\beta_{i}e_{i+2}) ]\Bigl) \Bigl)  \\ =  \frac{1}{8}(7\alpha _{1}+\beta _{2})e_{5}+\frac{1}{2}(\alpha _{2}+\beta _{3})e_{6}+\frac{1}{4}(-\alpha _{3}+5\beta _{4})e_{7}+\frac{1}{2}(-3\alpha _{4}+5\beta _{5})e_{8}+                              \\ \frac{1}{48} \sum\limits_{i=5}^{n-4} (i+1)(i^2+2i)\beta_{i+1}e_{i+4} - \frac{1}{48} \sum\limits_{i=5}^{n-4} (i^3+3i^2+2i-48)\alpha_ie_{i+4}  \Bigl)
\end{aligned}
\end{equation}

According to the structure of \(W^+(n)\), we have \([e_2, e_3] = e_5\). Applying the definition of a \(\frac{1}{2}\)-derivation yields 

\begin{equation}\label{varpi(e_5)_2 w+n}
\begin{aligned}
    \varphi(e_5)=\varphi ([e_{2},e_{3}])=\frac{1}{2}([\varphi (e_{2}),e_{3}]+[e_{2},\varphi (e_{3})])= \\
    =\frac{1}{2} \Bigl([\sum\limits_{i=1}^{n}\beta_ie_i,e_3]+[e_2,\frac{1}{2}(\alpha_1 e_3-\sum\limits_{i=3}^{n-2} (i-2)\alpha_{i}e_{i+2}+\sum\limits_{i=2}^{n-1}(i-1) \beta_{i}e_{i+1})]\Bigl) \\ = \beta_1 e_4+\frac{1}{4} (\alpha_1 + 3 \beta_2) e_5 + \beta_3 e_6 + \frac{1}{4} (-3 \alpha_3 + 7 \beta_4) e_7 + (-2 \alpha_4 + 3 \beta_5) e_8
\\ +\frac{1}{4} \sum\limits_{i=5}^{n-4} (i^2-2i+4)\beta_{i+1}e_{i+4} - \frac{1}{4} \sum\limits_{i=5}^{n-4} i(i-2)\alpha_ie_{i+4}  
\end{aligned}
\end{equation}

By equating the coefficients of the basis elements in the expressions for \(\varphi(e_5)\) in equations \eqref{varpi(e_5)_1 w+n} and \eqref{varpi(e_5)_2 w+n}, we obtain.

$$\begin{cases}
   e_4: \quad \beta_1=0   \\
   e_5: \quad  \frac{7}{8}\alpha_1+\frac{1}{8}\beta_2=\frac{3}{4}\beta_2+\frac{1}{4}\alpha_1   \\ 
    e_6: \quad \frac{1}{2}\alpha_2+\frac{1}{2}\beta_3=\beta_3  \\
    e_7: \quad \frac{1}{4}(-\alpha _{3}+5\beta _{4})=\frac{1}{4} (-3 \alpha_3 + 7 \beta_4)  \\
    e_8: \quad \frac{1}{2}(-3\alpha _{4}+5\beta _{5})=(-2 \alpha_4 + 3 \beta_5)  \\
    \frac{1}{48} \sum\limits_{i=5}^{n-4} \Bigl( (i+1)(i^2+2i)\beta_{i+1} -  (i^3+3i^2+2i-48)\alpha_i \Bigl)e_{i+4} = \frac{1}{4} \sum\limits_{i=5}^{n-4} \Bigl( (i^2-2i+4)\beta_{i+1} -  i(i-2)\alpha_i \Bigl) e_{i+4}  
\end{cases}$$

From this system, we obtain the following result: $b_1=0$
and  $b_{i+1}=a_i$ where $1 \le i \le n-4$.

We prove the following equality for $5\leq k\leq n-1$ by induction:
$$\varphi(e_k)=\sum\limits_{i=1}^{n-k+1}\alpha_ie_{i+k-1}, \  5\leq k\leq n-1$$

if $k=5$ 
$$\varphi(e_5)=\sum\limits_{i=1}^{n-4}\alpha_ie_{i+4}$$

The base case \(k=5\) is immediately satisfied by the relation above. Now, assuming the hypothesis holds for \(k\), we extend the argument to \(k+1\). By evaluating the \(\frac{1}{2}\)-derivation property on the elements \(e_{1}\) and \(e_{k}\), we obtain

\begin{center}
$\varphi(e_{k+1})=\varphi(\frac{1}{k-1}[e_1,e_k])=\frac{1}{2(k-1)}([\varphi(e_1),e_k]+[e_1,\varphi(e_k)])=\frac{1}{2(k-1)}\Big([\sum\limits_{i=1}^{n}\alpha_ie_i,e_k]+[e_1,\sum\limits_{i=1}^{n-k+1}\alpha_ie_{i+k-1}]\Big)$
\end{center}
\begin{center}
$=\sum\limits_{i=1}^{n-k}\alpha_ie_{i+k}$
\end{center}

Now, consider the condition of $\frac 1 2$-derivation for the elements $e_1, e_{n-1}:$
\begin{center}
$\varphi(e_{n})=\varphi([e_1,e_{n-1}])=\frac{1}{2(n-2)}([\varphi(e_1),e_{n-1}]+[e_1,\varphi(e_{n-1})])=\frac{1}{2(n-2)}\Big([\sum\limits_{i=1}^{n}\alpha_ie_i,e_{n-1}]+[e_1,\sum\limits_{i=1}^{2}\alpha_ie_{i+n-2}  ]\Big)$
\end{center}
\begin{center}
$= \frac{1}{2(n-2)}\Big( (n-2) \alpha_1 e_{n} +(n-2) \alpha_1 e_{n} \Big)= \alpha_1 e_{n}
$
\end{center}
Thus, we obtain that
\begin{equation}\label{varfi en w+n}
\varphi(e_{n})=\alpha_1 e_{n}
\end{equation}

\end{proof}

\subsection{\texorpdfstring{$\frac{1}{2}$-derivation of the algebra $m_{0,1}(2k+1)$}{1/2 derivations of the algebra m01}.}\label{subsec5}

\begin{Th}\label{THM: halfder m01n} Any $\frac{1}{2}$-derivation of the algebra $m_{0,1}(n)$
has the form

\begin{equation}\label{halfder: m01n}
\begin{cases}
\varphi(e_1)=\sum\limits_{i=1}^{n}\alpha_ie_i, \\ \varphi(e_2)=\alpha_1 e_2+\sum\limits_{i=3}^{n}\beta_ie_i,, \\  
\varphi(e_k)=\alpha_1 e_k+\frac{1}{2^{k-2}}\sum\limits_{i=3}^{n-k+2}\beta_{i}e_{i+k-2}+\frac{1}{2}(-1)^{k+1}a_{n-k+1}e_n, \  3\leq k\leq n-1,\\
\varphi(e_n)=\alpha_1 e_n. 
\end{cases}
\end{equation}
\end{Th}

\begin{proof}
    The proof for the algebra $m_0(n)$ is identical to the initial proof up to the 
basis element $e_{n-1}$. Due to the product 
$[e_l, e_{2k-l+1}] = (-1)^{l+1} e_{2k+1},$
only the coefficient preceding the basis element $e_n$ undergoes a change. 
In the general case, $\varphi(e_k)$ is given by:

$$\varphi(e_k)=\frac{2^{k-2}-1}{2^{k-2}}\alpha_1 e_k+\frac{1}{2^{k-2}}\sum\limits_{i=2}^{n-k+2}\beta_{i}e_{i+k-2}+\frac{1}{2}(-1)^{k+1}a_{n-k+1}e_n, \  3\leq k\leq n-1$$

Furthermore, Theorem~\eqref{THM: halfder m01n} differs from Theorem~\eqref{THM: halfder m0n} in the case where $\beta_2 = \alpha_1$. 
The proof for the condition $\beta_2 = \alpha_1$ is provided as follows:

In the current case, we consider the generation of $\varphi(e_n)$ in the following two ways:
\begin{enumerate}
    \item $\varphi(e_n) = \varphi([e_1, e_{n-1}])$,
    \item $\varphi(e_n) = -\varphi([e_2, e_{n-2}])$.
\end{enumerate}

\begin{equation}\label{m01n e_n type1}
\begin{aligned}
    \varphi(e_n) = \varphi([e_1, e_{n-1}])=\frac{1}{2}([\varphi(e_1),e_{n-1}]+[e_1,\varphi(e_{n-1})]) \\ =\frac{1}{2}([\sum\limits_{i=1}^{n}\alpha_ie_i,e_{n-1}]+[e_1,\frac{2^{n-3}-1}{2^{n-3}}\alpha_1 e_{n-1}+\frac{1}{2^{n-3}}\sum\limits_{i=2}^{3}\beta_{i}e_{i+n-3}+\frac{1}{2}a_{2}e_n])  = \frac{(2^{n-2}-1)\alpha_1+\beta_2}{2^{n -2}} e_n
\end{aligned}
\end{equation}
by second way

\begin{equation}\label{m01n e_n type2}
\begin{aligned}
    \varphi(e_n) = -\varphi([e_2, e_{n-2}])=-\frac{1}{2}([\varphi(e_2),e_{n-2}]+[e_2,\varphi(e_{n-2})]) \\ =-\frac{1}{2}([\sum\limits_{i=1}^{n}\beta_ie_i,e_{n-2}]+[e_2,\frac{2^{n-4}-1}{2^{n-4}}\alpha_1 e_{n-2}+\frac{1}{2^{n-4}}\sum\limits_{i=2}^{4}\beta_{i}e_{i+n-4}-\frac{1}{2}a_{3}e_n])  = \frac{(2^{n-4}-1)\alpha_1+(2^{n-4}+1)\beta_2}{2^{n -3}} e_n
\end{aligned}
\end{equation}

Equating the coefficients of the basis element $e_n$ in equations \eqref{m01n e_n type1} and \eqref{m01n e_n type2}, we obtain the following:

$$\frac{(2^{n-2}-1)\alpha_1+\beta_2}{2^{n -2}}=\frac{(2^{n-4}-1)\alpha_1+(2^{n-4}+1)\beta_2}{2^{n -3} }  \Longrightarrow \beta_2=\alpha_1  $$

 \end{proof}

\subsection{\texorpdfstring{$\frac{1}{2}$-derivation of the algebra $m_{0,2}(2k+2)$}{1/2 derivations of the algebra m02}.}\label{subsec6}

\begin{Th}\label{THM: halfder m02n} Any $\frac{1}{2}$-derivations of the algebra $m_{0,2}(n)$
has the form

\begin{equation}\label{halfder: m02n}
\begin{cases}
\varphi(e_1)= \alpha_1e_1+ \sum\limits_{i=3}^{n}\alpha_ie_i, \\ \varphi(e_2)=\alpha_1 e_2+\sum\limits_{i=4}^{n}\beta_ie_i,, \\
\varphi(e_3)=\alpha_1e_3+\frac{1}{2}\sum\limits_{i=4}^{n-1}\beta_ie_{i+1}+\frac{1}{2}a_{n-3}e_{n-1}+\frac{1}{2}(\frac{n}{2}-2)a_{n-2}e_n    \\
\varphi(e_k)=\alpha_1 e_k+\frac{1}{2^{k-2}}\sum\limits_{i=4}^{n-k+2}\beta_{i}e_{i+k-2}+\frac{1}{2}(-1)^{k+1}a_{n-k}e_{n-1}+\frac{1}{2}(-1)^{k+1}(\frac{n+1}{2}-k)a_{n-k+1}e_n, \  4\leq k\leq n-2,\\
\varphi(e_{n-1})=\alpha_1 e_{n-1}, \\
\varphi(e_{n})=\alpha_1e_n
\end{cases}
\end{equation}
\end{Th}

\begin{proof}

The proof of this case is identical to that of the algebra $m_0(n)$ up to the 
first $n-2$ basis elements. Due to the product 
$[e_j, e_{2k-j+2}] = (-1)^{j+1}(k-j+1)e_{2k+2},$
only the coefficient preceding the basis element $e_n$ undergoes a change. 
In the algebra $m_{0,1}(2k+1)$, the product 
$[e_l, e_{2k-l+1}] = (-1)^{l+1}e_{2k+1}$
affects only the coefficients preceding the basis element $e_n$, whereas for $m_{0,2}(2k+2)$, 
this product shifts to the coefficient preceding the basis element $e_{n-1}$. 
That is, the dimension of the algebra $m_{0,2}(2k+2)$ is one unit larger than that of the 
algebra $m_{0,1}(2k+1)$, and the principal minor of order $2k+1$ of the $1/2$-derivation matrix 
of $m_{0,2}(2k+2)$ coincides with the $1/2$-derivation matrix of $m_{0,1}(2k+1)$. 
According to Theorem~\eqref{THM: halfder m01n}, we have $\beta_2 = \alpha_1$. 
In the general case for $m_{0,2}(2k+2)$, $\varphi(e_k)$ is given by:

$$\varphi(e_k)=\alpha_1 e_k+\frac{1}{2^{k-2}}\sum\limits_{i=3}^{n-k+2}\beta_{i}e_{i+k-2}+\frac{1}{2}(-1)^{k+1}a_{n-k}e_{n-1}+\frac{1}{2}(-1)^{k+1}(\frac{n+1}{2}-k)a_{n-k+1}e_n , \  4\leq k\leq n-2$$

Furthermore, Theorem~\eqref{THM: halfder m02n} differs from Theorem~\eqref{THM: halfder m01n} in the case where $\beta_3=0 ;
\quad \alpha_2=0$. 
The proof for the condition $\beta_3=0 ;
\quad \alpha_2=0$ is provided as follows:

In the current case, we consider the generation of $\varphi(e_{n-1})$ in the following three ways:
\begin{enumerate}
    \item $\varphi(e_{n-1}) = \varphi([e_1, e_{n-2}])$,
    \item $\varphi(e_{n-1}) = -\varphi([e_2, e_{n-3}])$.
    \item $\varphi(e_{n-1}) =  \varphi([e_3, e_{n-4}])$.
\end{enumerate}

\begin{equation}\label{m02n e_n-1 type1}
\begin{aligned}
    \varphi(e_{n-1}) = \varphi([e_1, e_{n-2}])=\frac{1}{2}([\varphi(e_1),e_{n-2}]+[e_1,\varphi(e_{n-1})]) \\ =\frac{1}{2}([\sum\limits_{i=1}^{n}\alpha_ie_i,e_{n-1}]+[e_1,\alpha_1 e_{n-2}+\frac{1}{2^{n-4}}\sum\limits_{i=3}^{4}\beta_{i}e_{i+n-4}-\frac{1}{2}\alpha_{2}e_{n-1}+\frac{1}{4}(5-n)\alpha_3 e_n]) \\= \alpha_1 e_{n-1}+\Bigl(\frac{1}{2^{n-3}}\beta_3-\frac{1}{4}(n-3) \alpha_2 \Bigl)e_n
\end{aligned}
\end{equation}
by second way

\begin{equation}\label{m02n e_n-1 type2}
\begin{aligned}
    \varphi(e_{n-1}) = -\varphi([e_2, e_{n-3}])=-\frac{1}{2}([\varphi(e_2),e_{n-3}]+[e_2,\varphi(e_{n-3})]) \\ =-\frac{1}{2}([\sum\limits_{i=2}^{n}\beta_ie_i,e_{n-2}]+[e_2,\frac{2^{n-5}-1}{2^{n-5}}\alpha_1 e_{n-3}+\frac{1}{2^{n-5}}\sum\limits_{i=2}^{5}\beta_{i}e_{i+n-5}-\frac{1}{2}\alpha_{3}e_{n-1}+\frac{1}{2}(\frac{n+1}{2}-4)\alpha_3e_n]) \\ = \alpha_1 e_{n-1}+\Bigl(\frac{1}{2^{n-4}}(\frac{n}{2}-2)\beta_3-(\frac{n}{2}-3) \frac{\beta_3}{2} \Bigl)e_n
\end{aligned}
\end{equation}
next way
\begin{equation}\label{m02n e_n-1 type3}
\begin{aligned}
    \varphi(e_{n-1}) = \varphi([e_3, e_{n-4}])=\frac{1}{2}([\varphi(e_3),e_{n-4}]+[e_3,\varphi(e_{n-4})]) \\ =\frac{1}{2}([\alpha_1e_3+\frac{1}{2}\sum\limits_{i=3}^{n-1}\beta_ie_{i+1}+\frac{1}{2}a_{n-3}e_{n-1}+\frac{1}{2}(\frac{n}{2}-2)a_{n-2}e_n ,e_{n-4}]) \\ +\frac{1}{2}([e_3,\alpha_1 e_{n-4}+\frac{1}{2^{n-6}}\sum\limits_{i=3}^{6}\beta_{i}e_{i+n-6}-\frac{1}{2}\alpha_{4}e_{n-1}+\frac{1}{2}(\frac{n+1}{2}-5)\alpha_5e_n]) \\ = \alpha_1 e_{n-1}+\Bigl(\frac{1}{2^{n-5}}(\frac{n}{2}-3)\beta_3-(\frac{n}{2}-4) \frac{\beta_3}{4} \Bigl)e_n
\end{aligned}
\end{equation}

Equating the coefficients of the basis element $e_{n-1}$ in equations \eqref{m02n e_n-1 type1} , \eqref{m02n e_n-1 type2} \eqref{m02n e_n-1 type3} and, we obtain the following:

$$\frac{1}{2^{n-4}}\beta_3-\frac{1}{4}(n-3) \alpha_2 =\frac{1}{2^{n-4}}(\frac{n}{2}-2)\beta_3-(\frac{n}{2}-3) \frac{\beta_3}{2} =\frac{1}{2^{n-5}}(\frac{n}{2}-3)\beta_3-(\frac{n}{2}-4) \frac{\beta_3}{4}    $$
$$\Longrightarrow  \beta_3=0 ;
\quad \alpha_2=0$$

so

$$\varphi(e_{n-1})=\alpha_1 e_{n-1}$$

 \end{proof}

\subsection{\texorpdfstring{$\frac{1}{2}$-derivation of the algebra $m_{0,3}(2k+3)$}{1/2 derivations of the algebra m03}.}\label{subsec7}

\begin{Th}\label{THM: halfder m03n} Any $\frac{1}{2}$-derivation of the algebra $m_{0,3}(n)$
has the form

\begin{equation}\label{halfder: m03n}
\begin{cases}
\varphi(e_1)=\alpha_1e_1+\sum\limits_{i=4}^{n}\alpha_ie_i, \\ \varphi(e_2)=\alpha_1 e_2+\sum\limits_{i=5}^{n}\beta_ie_i,, \\ 
\varphi(e_3)=\alpha_1e_3+\frac{1}{2}\sum\limits_{i=5}^{n-1}\beta_ie_{i+1}+\frac{1}{2}a_{n-4}e_{n-2}+\frac{1}{2}(\frac{n-1}{2}-2)a_{n-3}e_{n-1} ,  \\
\varphi(e_4)=\alpha_1e_4+\frac{1}{2}\sum\limits_{i=5}^{n-2}\beta_ie_{i+2}-\frac{1}{2}a_{n-5}e_{n-2}-\frac{1}{2}(\frac{n-1}{2}-3)a_{n-4}e_{n-1}+(\frac{n-5}{2})a_{n-3}e_n ,  \\
\varphi(e_k)=\alpha_1 e_k+\frac{1}{2^{k-2}}\sum\limits_{i=5}^{n-k+2}\beta_{i}e_{i+k-2}+\frac{1}{2}(-1)^{k+1}a_{n-k-1}e_{n-2}+ \frac{1}{2}(-1)^{k+1} \Bigl(\frac{n}{2}-k  \Bigl) a_{n-k}e_{n-1} \\ \quad +\frac{1}{2}(-1)^{k}\Bigl( (k-3)\frac{n-3}{2}-\frac{(k-2)^2}{2}+\frac{n-2}{4}  \Bigl) a_{n-k+1}e_{n}, \  5\leq k\leq n-3,\\
\varphi(e_{n-2})=\alpha_1 e_{n-2}, \\
\varphi(e_{n-1})=\alpha_1 e_{n-1}, \\
\varphi(e_{n})=\alpha_1e_n
\end{cases}
\end{equation}
\end{Th}

\begin{proof}

We proceed with the proof of this theorem by employing the method established in Theorem~\eqref{THM: halfder m02n}

In the general case for $m_{0,3}(n)$, $\varphi(e_k)$ is given by:

$$\varphi(e_k)=\alpha_1 e_k+\frac{1}{2^{k-2}}\sum\limits_{i=4}^{n-k+2}\beta_{i}e_{i+k-2}+\frac{1}{2}(-1)^{k+1}a_{n-k-1}e_{n-2}+\frac{1}{2}(-1)^{k+1}(\frac{n}{2}-k)a_{n-k}e_{n-1}\  $$
$$+\frac{1}{2}(-1)^k\Bigl((k-3)\frac{n-3}{2}-\frac{(k-2)^2}{2}+ \frac{n-2}{4}\Bigl)\alpha_{n-k+1}e_n,  \quad 5\leq k\leq n-3$$

in the case where $\beta_4=0 ;
\quad \alpha_3=0$. 
The proof for the condition $\beta_4=0 ;
\quad \alpha_3=0$ is provided as follows:

In the current case, we consider the generation of $\varphi(e_{n-2})$ in the following three ways:
\begin{enumerate}
    \item $\varphi(e_{n-2}) = \varphi([e_1, e_{n-3}])$,
    \item $\varphi(e_{n-2}) = -\varphi([e_2, e_{n-4}])$,
    \item $\varphi(e_{n-2}) = \varphi([e_3, e_{n-5}])$.
\end{enumerate}

\begin{equation}\label{m03n e_n-2 type1}
\begin{aligned}
    \varphi(e_{n-2}) = \varphi([e_1, e_{n-3}])=\frac{1}{2}([\varphi(e_1),e_{n-3}]+[e_1,\varphi(e_{n-3})]) \\ =\frac{1}{2}([\sum\limits_{i=1}^{n}\alpha_ie_i,e_{n-1}]+[e_1,\alpha_1 e_{n-3}+\frac{1}{2^{n-5}}\sum\limits_{i=4}^{5}\beta_{i}e_{i+n-5}-\frac{1}{2}\alpha_{3}(3-\frac{n}{2})e_{n-1}+\frac{1}{2}(\frac{3n-16}{4})\alpha_4 e_n]) \\= \alpha_1 e_{n-2}+\Bigl(\frac{1}{2^{n-4}}\beta_4-\frac{1}{8}(n-4) \alpha_3 \Bigl)e_n
\end{aligned}
\end{equation}
by second way

\begin{equation}\label{m03n e_n-2 type2}
\begin{aligned}
    \varphi(e_{n-2}) = -\varphi([e_2, e_{n-4}])=-\frac{1}{2}([\varphi(e_2),e_{n-4}]+[e_2,\varphi(e_{n-4})]) \\ =-\frac{1}{2}([\alpha_1e_2+\sum\limits_{i=4}^{n}\beta_ie_i,e_{n-4}]+[e_2,\alpha_1 e_{n-4}+\frac{1}{2^{n-6}}\sum\limits_{i=4}^{6}\beta_{i}e_{i+n-6}+\frac{1}{2}\alpha_{3}e_{n-2}-\frac{1}{2}(\frac{n}{2}-4)\alpha_4e_{n-1}-(\frac{5n-32}{8})\alpha_3e_n]) \\ = \alpha_1 e_{n-2}+\frac{\beta_4}{2}(n-6) e_n
\end{aligned}
\end{equation}
next way
\begin{equation}\label{m03n e_n-2 type3}
\begin{aligned}
    \varphi(e_{n-2}) = \varphi([e_3, e_{n-5}])=\frac{1}{2}([\varphi(e_3),e_{n-5}]+[e_3,\varphi(e_{n-4})]) \\ =\frac{1}{2}([\alpha_1e_3+\frac{1}{2}\sum\limits_{i=3}^{n-1}\beta_ie_{i+1}+\frac{1}{2}a_{n-4}e_{n-2}+\frac{1}{2}(\frac{n-1}{2}-2)a_{n-3}e_n ,e_{n-5}]) \\ +\frac{1}{2}([e_3,\alpha_1 e_{n-5}+\frac{1}{2^{n-7}}\sum\limits_{i=4}^{7}\beta_{i}e_{i+n-7}-\frac{1}{2}\alpha_{4}e_{n-2}+\frac{1}{2}(\frac{n}{2}-5)\alpha_5e_{n-1}+(\frac{7n-52}{8})\alpha_4e_n]) \\ = \alpha_1 e_{n-2}-(\frac{3n-21}{8}+\frac{n-5}{2^{n-5}})\beta_4e_n
\end{aligned}
\end{equation}
Equating the coefficients of the basis element $e_n$ in equations \eqref{m03n e_n-2 type1} , \eqref{m03n e_n-2 type2} and \eqref{m03n e_n-2 type3}, we obtain the following:

$$\frac{1}{2^{n-4}}\beta_4-\frac{1}{8}(n-4) \alpha_3=-\frac{\beta_4}{2}(n-6)=-(\frac{3n-21}{8}+\frac{n-5}{2^{n-5}})\beta_4  \Longrightarrow \beta_4=\alpha_3=0  $$
so 

$$\varphi(e_{n-2})=\alpha_1 e_{n-2}$$

 \end{proof}

\subsection{\texorpdfstring{$\frac{1}{2}$-derivation of the one-parameter filiform Lie algebra}{1/2 derivations of the one-parameter filiform Lie algebra}.}\label{subsec8}

\begin{Th}\label{THM: halfder g7} The $\frac{1}{2}$-derivation of $g_{7,\alpha}$ ($\alpha \neq -2$)is given as follows:

\begin{equation}\label{halfder: g7}
\begin{cases}
\varphi(e_1)=\sum\limits_{i=1}^{7}\alpha_ie_i, \\ \varphi(e_2)=\alpha_1 e_2 + 4\alpha_2(2+\alpha)e_3 +\sum\limits_{i=4}^{7}\beta_ie_i, \\
\varphi(e_3) = \alpha_1 e_3 + 2\alpha_2(2+\alpha)e_4 + \frac{1}{2}\left[\beta_4 - \alpha_3(2+\alpha)\right]e_5 + \frac{1}{2}\left[\beta_5 - \alpha_4(2+\alpha)\right]e_6 + \frac{1}{2}\left[\beta_6 - \alpha_5(1+\alpha)\right]e_7 \\
\varphi(e_4) = \alpha_1 e_4 + \frac{3}{2}\alpha_2(2+\alpha)e_5 + \frac{1}{4}\left[\beta_4 - \alpha_3(2+\alpha)\right]e_6 + \frac{1}{4}\left[\beta_5 - \alpha_4(4+\alpha)\right]e_7 \\
\varphi(e_5) = \alpha_1 e_5 + \alpha_2(2+\alpha)e_6 + \frac{1}{8}\left[\beta_4 + \alpha_3(6-\alpha)\right]e_7 \\
\varphi(e_6) = \alpha_1 e_6 + \alpha_2(3+2\alpha)e_7 \\
\varphi(e_7)=\alpha_1 e_7. 
\end{cases}
\end{equation}

\end{Th}

\begin{Th}\label{THM: halfder g8} The $\frac{1}{2}$-derivation of $g_{8,\alpha}$ ($\alpha \neq -2$) is given as follows:

\begin{equation}\label{halfder: g8}
\begin{cases}
\varphi(e_1)=\sum\limits_{i=1}^{8}\alpha_ie_i, \\ \varphi(e_2)=\alpha_1 e_2 + 4\alpha_2(2+\alpha)e_3 +\sum\limits_{i=4}^{8}\beta_ie_i, \\
\varphi(e_3) = \varphi(e_3) = \alpha_1 e_3 + 2\alpha_2(2+\alpha)e_4 + \frac{1}{2}\left[\beta_4 - \alpha_3(2+\alpha)\right]e_5 + \frac{1}{2}\left[\beta_5 - \alpha_4(2+\alpha)\right]e_6 + \\ \quad + \frac{1}{2}\left[\beta_6 - \alpha_5(1+\alpha)\right]e_7 + \frac{1}{2}\left[\beta_7 - \alpha\alpha_6\right]e_8 \\
\varphi(e_4) = \alpha_1 e_4 + \frac{3}{2}\alpha_2(2+\alpha)e_5 + \frac{1}{4}\left[\beta_4 - \alpha_3(2+\alpha)\right]e_6 + \frac{1}{4}\left[\beta_5 - \alpha_4(4+\alpha)\right]e_7 + \frac{1}{4}\left[\beta_6 - \alpha_5(3+\alpha)\right]e_8 \\
\varphi(e_5) = \alpha_1 e_5 + \alpha_2(2+\alpha)e_6 + \frac{1}{8}\left[\beta_4 + \alpha_3(6-\alpha)\right]e_7 + \frac{1}{8}\left[\beta_5 - \alpha_4(4+\alpha)\right]e_8 \\
\varphi(e_6) = \alpha_1 e_6 + \alpha_2(3+2\alpha)e_7 + \frac{1}{16}\left[\beta_4 + \alpha_3(14-\alpha)\right]e_8 \\
\varphi(e_7) = \alpha_1 e_7 + \alpha_2(1+2\alpha)e_8 \\
\varphi(e_8)=\alpha_1 e_8. 
\end{cases}
\end{equation}

\end{Th}

\begin{Th}\label{THM: halfder g9} 
The \(\frac{1}{2}\)-derivation of \(g_{9,\alpha }\) (\(\alpha \neq -\frac{5}{2}\)) is given as follows:
\begin{equation}\label{halfder: g9} 
\begin{cases} \varphi(e_1)=\sum\limits_{i=1}^{9}\alpha_ie_i, \\  \varphi(e_2)=\alpha_1 e_2 + 4\alpha_2(2+\alpha)e_3 +\sum\limits_{i=4}^{9}\beta_ie_i, \\ \varphi(e_3) = \alpha_1 e_3 + 2\alpha_2(2+\alpha)e_4 + \frac{1}{2}\left[\beta_4 - \alpha_3(2+\alpha)\right]e_5 + \frac{1}{2}\left[\beta_5 - \alpha_4(2+\alpha)\right]e_6 + \\  \quad + \frac{1}{2}\left[\beta_6 - \alpha_5(1+\alpha)\right]e_7 + \frac{1}{2}\left[\beta_7 - \alpha\alpha_6\right]e_8 + \frac{1}{2}\left[\beta_8 - \frac{2\alpha^2+3\alpha-2}{2\alpha+5}\alpha_7\right]e_9, \\ \varphi(e_4) = \alpha_1 e_4 + \frac{3}{2}\alpha_2(2+\alpha)e_5 + \frac{1}{4}\left[\beta_4 - \alpha_3(2+\alpha)\right]e_6 + \frac{1}{4}\left[\beta_5 - \alpha_4(4+\alpha)\right]e_7 + \\  \quad + \frac{1}{4}\left[\beta_6 - \alpha_5(3+\alpha)\right]e_8 + \frac{1}{4}\left[\beta_7 - \alpha_6(1+2\alpha)\right]e_9, \\ \varphi(e_5) = \alpha_1 e_5 + \alpha_2(2+\alpha)e_6 + \frac{1}{8}\left[\beta_4 + \alpha_3(6-\alpha)\right]e_7 + \frac{1}{8}\left[\beta_5 - \alpha_4(4+\alpha)\right]e_8 + \\  \quad + \frac{1}{8}\left[\beta_6 + \alpha_5(10-3\alpha)\right]e_9, \\ \varphi(e_6) = \alpha_1 e_6 + \alpha_2(3+2\alpha)e_7 + \frac{1}{16}\left[\beta_4 + \alpha_3(14-\alpha)\right]e_8 + \frac{1}{16}\left[\beta_5 + \alpha_4(12-3\alpha)\right]e_9, \\ \varphi(e_7) = \alpha_1 e_7 + \alpha_2(1+2\alpha)e_8 + \frac{1}{32}\left[\beta_4 + \alpha_3(30-5\alpha)\right]e_9, \\ \varphi(e_8) = \alpha_1 e_8 + \alpha_2(1+4\alpha)e_9, \\ \varphi(e_9)=\alpha_1 e_9.  \end{cases}
\end{equation}
\end{Th}

\begin{Th}\label{THM: halfder g10}
The \(\frac{1}{2}\)-derivation of \(g_{10,\alpha }\) (\(\alpha \neq -\frac{5}{2}\)) is given as follows:
\begin{equation}\label{halfder: g10}
\begin{cases} \varphi(e_1)=\sum\limits_{i=1}^{10}\alpha_ie_i, \\  \varphi(e_2)=\alpha_1 e_2 + 4\alpha_2(2+\alpha)e_3 +\sum\limits_{i=4}^{10}\beta_ie_i, \\ \varphi(e_3) = \alpha_1 e_3 + 2\alpha_2(2+\alpha)e_4 + \frac{1}{2}\left[\beta_4 - \alpha_3(2+\alpha)\right]e_5 + \frac{1}{2}\left[\beta_5 - \alpha_4(2+\alpha)\right]e_6 + \\  \quad + \frac{1}{2}\left[\beta_6 - \alpha_5(1+\alpha)\right]e_7 + \frac{1}{2}\left[\beta_7 - \alpha\alpha_6\right]e_8 + \frac{1}{2}\left[\beta_8 - \frac{2\alpha^2+3\alpha-2}{2\alpha+5}\alpha_7\right]e_9 + \frac{1}{2}\left[\beta_9 - \frac{2\alpha^2+\alpha-1}{2\alpha+5}\alpha_8\right]e_{10}, \\ \varphi(e_4) = \alpha_1 e_4 + \frac{3}{2}\alpha_2(2+\alpha)e_5 + \frac{1}{4}\left[\beta_4 - \alpha_3(2+\alpha)\right]e_6 + \frac{1}{4}\left[\beta_5 - \alpha_4(4+\alpha)\right]e_7 + \\  \quad + \frac{1}{4}\left[\beta_6 - \alpha_5(3+\alpha)\right]e_8 + \frac{1}{4}\left[\beta_7 - \alpha_6(1+2\alpha)\right]e_9 + \frac{1}{4}\left[\beta_8 - \alpha_7\frac{2\alpha^2+5\alpha+1}{2\alpha+5}\right]e_{10}, \\ \varphi(e_5) = \alpha_1 e_5 + \alpha_2(2+\alpha)e_6 + \frac{1}{8}\left[\beta_4 + \alpha_3(6-\alpha)\right]e_7 + \frac{1}{8}\left[\beta_5 - \alpha_4(4+\alpha)\right]e_8 + \\  \quad + \frac{1}{8}\left[\beta_6 + \alpha_5(10-3\alpha)\right]e_9 + \frac{1}{8}\left[\beta_7 + \alpha_6\frac{20-2\alpha^2-3\alpha}{2\alpha+5}\right]e_{10}, \\ \varphi(e_6) = \alpha_1 e_6 + \alpha_2(3+2\alpha)e_7 + \frac{1}{16}\left[\beta_4 + \alpha_3(14-\alpha)\right]e_8 + \frac{1}{16}\left[\beta_5 + \alpha_4(12-3\alpha)\right]e_9 + \\ \quad + \frac{1}{16}\left[\beta_6 + \alpha_5\frac{3\alpha^2-4\alpha-44}{2\alpha+5}\right]e_{10}, \\ \varphi(e_7) = \alpha_1 e_7 + \alpha_2(1+2\alpha)e_8 + \frac{1}{32}\left[\beta_4 + \alpha_3(30-5\alpha)\right]e_9 + \frac{1}{32}\left[\beta_5 + \alpha_4\frac{92+12\alpha-5\alpha^2}{2\alpha+5}\right]e_{10}, \\ \varphi(e_8) = \alpha_1 e_8 + \alpha_2(1+4\alpha)e_9 + \frac{1}{64}\left[\beta_4 + \alpha_3\frac{124+42\alpha-7\alpha^2}{2\alpha+5}\right]e_{10}, \\ \varphi(e_9) = \alpha_1 e_9 + \alpha_2(1+6\alpha)e_{10}, \\ \varphi(e_{10}) = \alpha_1 e_{10}. 
\end{cases} 
\end{equation}
\end{Th}

\begin{Th}\label{THM: halfder g11} 
The \(\frac{1}{2}\)-derivation of \(g_{11,\alpha }\) (\(\alpha \neq -\frac{5}{2}, -1, -3\)) is given as follows:
\begin{equation}\label{halfder: g11}
\begin{cases} \varphi(e_1)=\sum\limits_{i=1}^{11}\alpha_ie_i, \\  \varphi(e_2)=\alpha_1 e_2 + 4\alpha_2(2+\alpha)e_3 +\sum\limits_{i=4}^{11}\beta_ie_i, \\
\varphi(e_3) = \alpha_1 e_3 + 2\alpha_2(2+\alpha)e_4 + \frac{1}{2}\left[\beta_4 - \alpha_3(2+\alpha)\right]e_5 + \frac{1}{2}\left[\beta_5 - \alpha_4(2+\alpha)\right]e_6 + \\  \quad + \frac{1}{2}\left[\beta_6 - \alpha_5(1+\alpha)\right]e_7 + \frac{1}{2}\left[\beta_7 - \alpha\alpha_6\right]e_8 + \frac{1}{2}\left[\beta_8 - \frac{2\alpha^2+3\alpha-2}{2\alpha+5}\alpha_7\right]e_9 + \\ \quad + \frac{1}{2}\left[\beta_9 - \frac{2\alpha^2+\alpha-1}{2\alpha+5}\alpha_8\right]e_{10} + \frac{1}{2}\left[\beta_{10} - \frac{2\alpha^3+2\alpha^2+3}{2(\alpha^2+4\alpha+3)}\alpha_9\right]e_{11}, \\ \varphi(e_4) = \alpha_1 e_4 + \frac{3}{2}\alpha_2(2+\alpha)e_5 + \frac{1}{4}\left[\beta_4 - \alpha_3(2+\alpha)\right]e_6 + \frac{1}{4}\left[\beta_5 - \alpha_4(4+\alpha)\right]e_7 + \\  \quad + \frac{1}{4}\left[\beta_6 - \alpha_5(3+\alpha)\right]e_8 + \frac{1}{4}\left[\beta_7 - \alpha_6(1+2\alpha)\right]e_9 + \frac{1}{4}\left[\beta_8 - \alpha_7\frac{2\alpha^2+5\alpha+1}{2\alpha+5}\right]e_{10} + \\ \quad + \frac{1}{4}\left[\beta_9 - \alpha_8\frac{2\alpha^3+4\alpha^2+5\alpha+4}{2(\alpha^2+4\alpha+3)}\right]e_{11}, \\ \varphi(e_5) = \alpha_1 e_5 + \alpha_2(2+\alpha)e_6 + \frac{1}{8}\left[\beta_4 + \alpha_3(6-\alpha)\right]e_7 + \frac{1}{8}\left[\beta_5 - \alpha_4(4+\alpha)\right]e_8 + \\  \quad + \frac{1}{8}\left[\beta_6 + \alpha_5(10-3\alpha)\right]e_9 + \frac{1}{8}\left[\beta_7 + \alpha_6\frac{20-2\alpha^2-3\alpha}{2\alpha+5}\right]e_{10} + \\ \quad + \frac{1}{8}\left[\beta_8 + \alpha_7\frac{14\alpha^2+32\alpha+23-2\alpha^3}{2(\alpha^2+4\alpha+3)(2\alpha+5)}\right]e_{11}, \\ \varphi(e_6) = \alpha_1 e_6 + \alpha_2(3+2\alpha)e_7 + \frac{1}{16}\left[\beta_4 + \alpha_3(14-\alpha)\right]e_8 + \frac{1}{16}\left[\beta_5 + \alpha_4(12-3\alpha)\right]e_9 + \\  \quad + \frac{1}{16}\left[\beta_6 + \alpha_5\frac{3\alpha^2-4\alpha-44}{2\alpha+5}\right]e_{10} + \frac{1}{16}\left[\beta_7 + \alpha_6\frac{2\alpha^3-26\alpha^2-40\alpha+4}{2(\alpha^2+4\alpha+3)(2\alpha+5)}\right]e_{11}, \\ \varphi(e_7) = \alpha_1 e_7 + \alpha_2(1+2\alpha)e_8 + \frac{1}{32}\left[\beta_4 + \alpha_3(30-5\alpha)\right]e_9 + \frac{1}{32}\left[\beta_5 + \alpha_4\frac{92+12\alpha-5\alpha^2}{2\alpha+5}\right]e_{10} + \\ \quad + \frac{1}{32}\left[\beta_6 + \alpha_5\frac{232+4\alpha^2-32\alpha-5\alpha^3}{2(\alpha^2+4\alpha+3)(2\alpha+5)}\right]e_{11}, \\ 
\varphi(e_8) = \alpha_1 e_8 + \alpha_2(1+4\alpha)e_9 + \frac{1}{64}\left[\beta_4 + \alpha_3\frac{124+42\alpha-7\alpha^2}{2\alpha+5}\right]e_{10} + \\ \quad + \frac{1}{64}\left[\beta_5 + \alpha_4\frac{412+134\alpha+14\alpha^2-7\alpha^3}{2(\alpha^2+4\alpha+3)(2\alpha+5)}\right]e_{11}, \\ \varphi(e_9) = \alpha_1 e_9 + \alpha_2(1+6\alpha)e_{10} + \frac{1}{128}\left[\beta_4 + \alpha_3\frac{444+214\alpha-3\alpha^2-9\alpha^3}{2(\alpha^2+4\alpha+3)(2\alpha+5)}\right]e_{11}, \\ \varphi(e_{10}) = \alpha_1 e_{10} + \alpha_2(1+8\alpha)e_{11}, \\ 
\varphi(e_{11}) = \alpha_1 e_{11}. 
\end{cases} 
\end{equation}
\end{Th}

\section{Local \texorpdfstring{$\frac{1}{2}$}{1/2} -derivation }\label{sec4}

\begin{defn}\label{defnlocal}
 A \textit{linear map} $\nabla$ \textit{is called a local $\frac{1}{2}$-derivation, if for any $x \in \mathfrak{L}$, there exists a $\frac{1}{2}$-derivation $\mathcal{D}_x : \mathfrak{L} \rightarrow \mathfrak{L}$ (depending on $x$) such that $\nabla(x) = \mathcal{D}_x(x)$. The set of all local $\frac{1}{2}$-derivations on $\mathfrak{L}$ we denote by $\mathrm{LocDer}_\frac{1}{2}(\mathfrak{L})$.}

\end{defn}

Let $B=(b_{i,j})$ be the matrix of $\nabla$ with respect to $\{e_1,\dots,e_n\}$. For $x=\sum\limits_{i=1}^n x_i e_i$, write $\bar{x}=(x_1,\dots,x_n)^T$. If $\mathcal{D}_x$ has matrix $A_x$, then the local condition is equivalent to the vector equation
\begin{equation}\label{systemlocal}
    B\bar{x}=A_x\bar{x}.
\end{equation}

\subsection{Local \texorpdfstring{$\frac{1}{2}$}{1/2}- derivation of the algebra \texorpdfstring{$m_0(n)$}{m(0n)}}\label{subsec9}

\begin{Th}\label{thm:local-m0}
A linear map $\nabla$ on the filiform Lie algebra $\mathfrak{m}_0(n)$ ($n \ge 3$) is a local $\frac{1}{2}$-derivation if and only if its matrix representation $M_{\nabla} = (b_{i,j})_{n \times n}$ with respect to the homogeneous basis $\{e_1, e_2, \dots, e_n\}$ is lower triangular and invertible, i.e.,
\begin{equation}\label{eq:local-matrix-form}
M_{\nabla_{m_0(n)}} = \begin{pmatrix}
b_{1,1} & 0 & 0 & 0 & \cdots & 0 & 0 \\
b_{2,1} & b_{2,2} & 0 & 0 & \cdots & 0 & 0 \\
b_{3,1} & b_{3,2} & b_{3,3} & 0 & \cdots & 0 & 0 \\
b_{4,1} & b_{4,2} & b_{4,3} & b_{4,4} & \cdots & 0 & 0 \\
\vdots & \vdots & \vdots & \vdots & \ddots & \vdots & 0 \\
b_{n-1,1} & b_{n-1,2} & b_{n-1,3} & b_{n-1,4} & \cdots & b_{n-1,n-1} & 0 \\
b_{n,1} & b_{n,2} & b_{n,3} & b_{n,4} & \cdots & b_{n,n-1} & b_{n,n}
\end{pmatrix},
\end{equation}
where $b_{i,j} \in \mathbb{C}$ .
\end{Th}

\begin{proof}
Let $\nabla $ be an arbitrary local $\frac{1}{2}$-derivation on $\mathfrak{m}_0(n)$, and, let $ M_{\nabla} $ be the matrix of $ \nabla$. By the definition for
all $x = x_1e_1 + \cdot\cdot\cdot + x_ne_n \in L_n$ there exists an $\frac{1}{2}$-derivation $\mathcal{D}_x $ on $ \mathfrak{m}_0(n) $ with respect to the basis $ \{e_1, e_2, . . . , e_n\} $ of $m_0(n)$
$$\nabla(x)=\mathcal{D}_x(x).$$

By Theorem~\eqref{THM: halfder m0n}  the $\frac{1}{2}$-derivation $\mathcal{D}_x$ has the following matrix form:
\[
A_x = \begin{pmatrix}\label{A_x form matrix m0n}
\alpha_1^x & 0 & 0 & 0 & \cdots & 0 & 0 \\
\alpha_2^x & \beta_2^x & 0 & 0 & \cdots & 0 & 0 \\
\alpha_3^x & \beta_3^x & \frac{\alpha_1^x+ \beta_2^x}{2} & 0 & \cdots & 0 & 0  \\
\alpha_4^x & \beta_4^x & \frac{\beta_3^x}{2}  & \frac{3\alpha_1^x+\beta_2^x}{4} & \cdots & 0 & 0 \\
\vdots & \vdots & \vdots & \vdots & \ddots & \vdots & 0 \\
\alpha_{n-1}^x & \beta_{n-1}^x &  \frac{1}{2}\beta_{n-2}^x &  \frac{1}{4} \beta_{n-3}^x & \cdots & \frac{(2^{n-3}-1)\alpha^x_1+\beta^x_2}{2^{n -3}} & 0 \\
\alpha_n^x & \beta_n^x & \frac{1}{2} \beta_{n-1}^x & \frac{1}{4} \beta_{n-2}^x & \cdots & \frac{\beta^x_3}{2^{n -3}}  & (\frac{(2^{n-2}-1)\alpha^x_1+\beta^x_2}{2^{n -2}} 
\end{pmatrix}.
\]
By systematically evaluating the local $\frac{1}{2}$-derivation criterion \(\nabla(x) = \mathcal{D}_x(x)\)—or its equivalent matrix representation \(M_\nabla\bar{x} = A_x\bar{x}\), where \(\bar{x}\) denotes the coordinate vector of \(x\)—at the basis elements \(x = e_1, \dots, e_n\), we deduce that the matrix \(A = (b_{i,j})\) possesses a lower triangular structure (\(b_{i,j} = 0\) whenever \(i < j\)) featuring strictly non-zero diagonal entries (\(b_{k,k} \neq 0\) for all \(1 \le k \le n\)). This architectural characteristic is deeply rooted in the algebraic properties of the underlying complex field \(\mathbb{C}\), which serves as the base field for the algebra \(\mathfrak{m}_0(n)\). Consequently, extending the relation \(A\bar{x} = A_x\bar{x}\) to a generic element \(x = \sum\limits_{i=1}^n x_i e_i\) yields the subsequent system of algebraic equations:

\begin{equation}\label{local sysytem m0n}
    \sum\limits_{j=1}^{i}b_{i,j}x_j=\alpha_i^x x_1+\beta_i^x x_2+\sum\limits_{j=3}^{i-1}\frac{1}{2^{j-2}} \beta_{i-j+2}^xx_j +\frac{(2^{i-2}-1)\alpha_1^x+\beta_2^x}{2^{i-2}}x_i
\end{equation}

in terms of the variables \(\alpha_{i}^{x}\) $(1 \leq i \leq n)$ and  \(\beta_{i}^{x}\) $(2 \leq i \leq n)$.

We shall now demonstrate that the linear operator represented by the matrix $A$
 a local $\frac{1}{2}$-derivation. By definition, if for every element $x \in \mathfrak{m}_0(n)$
there exists a matrix $A_x$ of the form given in Theorem~\eqref{THM: halfder m0n} satisfying the identity
$$A_x \bar{x} = A \bar{x}$$
thereby confirming that the corresponding linear operator is indeed a local $\frac{1}{2}$-derivation. Alternately, the linear operator determined by \(A\) represents a local $\frac{1}{2}$-derivation provided that, for each \(x \in \mathfrak{m}_0(n)\), the system of linear equations~\eqref{local sysytem m0n} is solvable with respect to the variables

$\alpha_1^x, \alpha_2^x, \dots, \alpha_n^x, \beta_2^x, \dots, \beta_n^x $.

To establish this, we analyze
the following cases:

\begin{itemize}
    \item  In the case where $x_1 \neq 0$, we find $\alpha_1^x = b_{1,1}$, whereas for any $2 \le i \le n$, the remaining components $\alpha_{i}^{x}$ are explicitly given by:
$$\alpha_i^x = b_{i,1} + \frac{1}{x_1}\Bigl( (b_{i,2}-\beta_i^x)x_2+\sum\limits_{j=3}^{i-1} (b_{i,j} - \sum\limits_{j=3}^{i-1}\frac{1}{2^{j-2}} \beta_{i-j+2}^x) x_j +\bigl(b_{i,i}-\frac{(2^{i-2}-1)\alpha_1^x+\beta_2^x}{2^{i-2}}\bigl)x_i \Bigl),$$
where $2 \le i \le n$, and these equalities give a solution of the system of equations~\eqref{local sysytem m0n}. Here  $\beta_{2}^x,\beta_{3}^x, \cdots , \beta_{n}^x$ are defined arbitrary.[Since the linear system~\eqref{local sysytem m0n} comprises $n$ equations and $2n-1$ variables, it possesses $n-1$ degrees of freedom. Consequently, these specific parameters may be selected arbitrarily with respect to the remaining components to guarantee solvability.
]
\item  if $x_1=0$ and $x_2 \neq 0$, then $b_{2,2}=\beta_2^x$ for each $i$ satisfying $3 \le i \le n$, the remaining unknowns $\beta_i^x$ can be recursively determined by the following explicit formula:

$$\beta_i^x = b_{i,2} + \frac{1}{x_2}\Bigl( \sum\limits_{j=3}^{i-1} (b_{i,j} - \frac{1}{2^{j-2}} \beta_{i-j+2}^x) x_j +\bigl(b_{i,i}-\frac{(2^{i-2}-1)\alpha_1^x+\beta_2^x}{2^{i-2}}\bigl)x_i\Bigl),$$
give solution of the system of equation~\eqref{local sysytem m0n}. Here $\alpha_1^x$ is arbitrary value.

\item If $ x_1=x_2=\cdots= x_{m-1} = 0$ and $x_m \neq 0, $ then the following equalities, $m \le i \le n$

$$\sum\limits_{j=m}^{i}b_{i,j}x_j=\sum\limits_{j=m}^{i-1}\frac{1}{2^{j-2}} \beta_{i-j+2}^xx_j +\frac{(2^{i-2}-1)\alpha_1^x+\beta_2^x}{2^{i-2}}x_i$$
so

$$\beta_2^x=2^{m-2} b_{m,m}-(2^{m-2}-1)\alpha_1^x$$

$$\cdots\cdots$$

$$\beta_{i-m+2}^x= 2^{m-2} b_{i,m}+\frac{2^{m-2}}{x_m}\Bigl( \sum\limits_{j=3}^{i-1} (b_{i,j} - \frac{1}{2^{j-2}} \beta_{i-j+2}^x) x_j +\bigl(b_{i,i}-\frac{(2^{i-2}-1)\alpha_1^x+\beta_2^x}{2^{i-2}}\bigl)x_i\Bigl)$$

give solution of the system of equation~\eqref{local sysytem m0n}. Here $\alpha_1^x$ is arbitrary value.
\end{itemize}
Consequently, the system of linear equations~\eqref{local sysytem m0n} consistently possesses a solution
for every element $x \in m_0(n)$, which implies that the linear operator induced by
the matrix $A$ constitutes a local $\frac{1}{2}$-derivation. Owing to the fact that the local
$\frac{1}{2}$-derivation $\nabla$ was chosen arbitrarily and is pointwise defined by the
$\frac{1}{2}$- derivation as characterized in Theorem~\eqref{THM: halfder m0n}, it follows that the matrix
representation of any local $\frac{1}{2}$-derivation of the algebra $m_0(n)$ necessarily admits
a lower triangular form.
\end{proof}

\subsection{Local \texorpdfstring{$\frac{1}{2}$}{1/2}- derivation of the algebra \texorpdfstring{$m_2(n)$}{m(2n)}}\label{subsec10}
The following theorems can be proven similar to the proof of Theorem~\eqref{thm:local-m0}.

\begin{Th}\label{thm:local-m2n}
Let $\nabla$ be a linear operator on the filiform Lie algebra $\mathfrak{m}_2(n)$ with $n \ge 5$. Then, $\nabla$ acts as a local $\frac{1}{2}$- derivations if and only if its matrix $M_{\nabla} = (b_{i,j})_{n \times n}$ with respect to the homogeneous basis $\{e_1, e_2, \dots, e_n\}$ exhibits an invertible lower triangular configuration, that is,

\begin{equation}\label{eq:local-matrix-form m_2(n)}
M_{\nabla} = \begin{pmatrix}
b_{1,1} & 0 & 0 & 0 & \cdots & 0 & 0 \\
b_{2,1} & b_{2,2} & 0 & 0 & \cdots & 0 & 0 \\
b_{3,1} & b_{3,2} & b_{3,3} & 0 & \cdots & 0 & 0 \\
b_{4,1} & b_{4,2} & b_{4,3} & b_{4,4} & \cdots & 0 & 0 \\
b_{5,1} & b_{5,2} & 0 & b_{5,4}& \cdots & 0 & 0 \\
\vdots & \vdots & \vdots & \vdots & \ddots & \vdots & 0 \\
b_{n-2,1} & b_{n-2,2} & 0 & 0 & \cdots & 0& 0 \\
b_{n-1,1} & b_{n-1,2} & b_{n-1,3} & 0 & \cdots & b_{n-1,n-1} & 0 \\
b_{n,1} & b_{n,2} & b_{n,3} & b_{n,4} & \cdots & b_{n,n-1} & b_{n,n}
\end{pmatrix},
\end{equation}
where $b_{i,j} \in \mathbb{C}$.
\end{Th}

\begin{proof}
Similar to the proof of
Theorem~\ref{thm:local-m0} we have

By Theorem~\ref{THM: halfder m2n}  the $\frac{1}{2}$- derivations $\mathcal{D}_x$ has the following matrix form:

\begin{equation}\label{A_x form matrix m0n}
A_x = \begin{pmatrix}
\alpha^x_1 & 0 & 0 & 0 & \cdots & 0 & 0 \\
\alpha^x_2 & \alpha^x_1 & 0 & 0 & \cdots & 0 & 0 \\
\alpha^x_3 & 6\alpha^x_2 & \alpha^x_1 & 0 & \cdots & 0 & 0 \\
\alpha^x_4 & \alpha^x_3 & 3\alpha^x_2 & \alpha^x_1 & \cdots & 0 & 0 \\
\alpha^x_5 & \alpha^x_4 & 0 & 2\alpha^x_2 & \cdots & 0 & 0 \\
\vdots & \vdots & \vdots & \vdots & \ddots & \vdots & \vdots \\
\alpha^x_{n-1} & \beta^x_{n-1} & \frac{1}{2}(\beta^x_{n-2}-\alpha^x_{n-3}) & 0 & \cdots & \alpha^x_1 & 0 \\
\alpha^x_n & \beta^x_n & \frac{1}{2}(\beta^x_{n-1}-\alpha^x_{n-2}) & \frac{1}{4}(\beta^x_{n-2}-\alpha^x_{n-3}) & \cdots & \frac{2^{n-5}+1}{2^{n-5}}\alpha^x_2 & \alpha^x_1
\end{pmatrix}
\end{equation}

Evaluating the local $\frac{1}{2}$- derivations condition $A\bar{x} = A_x\bar{x}$ successively at each basis vector $x = e_k$ ($1 \le k \le n$) demonstrates that the matrix $A = (b_{i,j})$ is lower triangular with non-vanishing diagonal entries. Such a configuration fundamentally relies on the algebraic closedness or specific properties of the complex field associated with the algebra $\mathfrak{m}_2(n)$. Extending this relation to an arbitrary linear combination $x = \sum\limits_{i=1}^n x_i e_i$ gives rise to the following system of algebraic equations:

\begin{equation}\label{local sysytem m2n}
    \begin{cases}
    b_{1,1}x_1=\alpha_1^x x_1 \\
    b_{2,1}x_1+b_{2,2}x_2=\alpha_2^x x_1+\alpha_1^x x_2 \\
     \sum\limits_{j=1}^{3}b_{3,j}x_j =\alpha_3^x x_1+6\alpha_2^x x_2 +\alpha_1^x x_3 , \\
     b_{i,1}x_1+b_{i,2}x_2+b_{i,i-1}x_{i-1}+b_{i,i}x_i=a_i^x x_1+a_{i-1}^x x_2+ \frac{2^{i-5}+1}{2^{i-5}} \alpha^x_2   x_{i-1}+\alpha_1^xx_i, \quad 4 \le i\le n-3, \\
      \sum\limits_{j=1}^{2}b_{n-2,j}x_j+ \sum\limits_{j=n-3}^{n-2}b_{n-2,j}x_{j}=\alpha_{n-2}^x x_1+\beta_{n-2}^x x_2+  \frac{2^{n-7}+1}{2^{n-7}} \alpha^x_2   x_{n-4}+\alpha_1^xx_{n-2},  \\
     \sum\limits_{j=1}^{3}b_{n-1,j}x_j+ \sum\limits_{j=n-2}^{n-1}b_{n-1,j}x_{j}=\alpha_{n-1}^x x_1+\beta_{n-1}^x x_2+ \frac{1}{2}(\beta^x_{n-2}-\alpha^x_{n-3})x_3+ \frac{2^{n-6}+1}{2^{n-6}} \alpha^x_2   x_{n-2}+\alpha_1^xx_{n-1},  \\
     \sum\limits_{j=1}^{4}b_{n,j}x_j+ \sum\limits_{j=n-1}^{n}b_{n,j}x_{j}=\alpha_{n}^x x_1+\beta_{n}^x x_2+ \frac{1}{2}(\beta^x_{n-1}-\alpha^x_{n-2})x_3 +\frac{1}{4}(\beta^x_{n-2}-\alpha^x_{n-3})x_4+ \frac{2^{n-5}+1}{2^{n-5}} \alpha^x_2   x_{n-1}+\alpha_1^xx_{n},  \\
    \end{cases}
\end{equation}

in terms of the variables \(\alpha_{i}^{x}\) $(1 \leq i \leq n)$ and  \(\beta_{i}^{x}\) $(n-2 \leq i \leq n)$.

Now, we prove that the linear operator, defined by the matrix A, is a local $\frac{1}{2}$- derivations. Similar to the proof of Theorem~\ref{thm:local-m0} we have
\begin{itemize}
\item \textbf{ $x_1 \neq 0$.}
We select the leading scaling parameter of the $\frac{1}{2}$- derivation to be $\alpha_1^x = b_{1,1} \neq 0$ 

$\alpha_2^x=b_{2,1}+\frac{x_2}{x_1}\bigl( b_{2,2}-b_{1,1}\bigl)$

$\alpha_3^x=\frac{1}{x_1}\bigl( \sum\limits_{j=1}^{3}b_{3,j}x_j -  6\Bigl( b_{2,1}+\frac{x_2}{x_1}\bigl( b_{2,2}-b_{1,1}\bigl) \Bigl) x_2 -b_{1,1} x_3  \bigl)      $

$\cdots \cdots \cdots \cdots$

$\alpha_i^x=b_{i,1}+\frac{1}{x_1} \Bigl( (b_{i,2}-a_{i-1}^x)x_2+ \bigl(b_{i,i-1}-\frac{2^{i-5}+1}{2^{i-5}} \alpha^x_2\bigl)x_{i-1} +(b_{i,i}-\alpha_1)x_i \Bigl)$ $4 \le i \le n-3$

$\alpha_{n-2}^x=\frac{1}{x_1}\Bigl( \sum\limits_{j=1}^{2}b_{n-2,j}x_j+ \sum\limits_{j=n-3}^{n-2}b_{n-2,j}x_{j}-\beta_{n-2}^x x_2-  \frac{2^{n-7}+1}{2^{n-7}} \alpha^x_2   x_{n-4}-\alpha_1^xx_{n-3}   \Bigl)$

$\alpha_{n-1}^x=\frac{1}{x_1}\Bigl( \sum\limits_{j=1}^{3}b_{n-1,j}x_j+ \sum\limits_{j=n-2}^{n-1}b_{n-1,j}x_{j}-\beta_{n-1}^x x_2- \frac{1}{2}(\beta^x_{n-2}-\alpha^x_{n-3})x_3- \frac{2^{n-6}+1}{2^{n-6}} \alpha^x_2   x_{n-2}-\alpha_1^xx_{n-1}   \Bigl)$

$\alpha_n^x=\frac{1}{x_1}\Bigl(\sum\limits_{j=1}^{4}b_{n,j}x_j+ \sum\limits_{j=n-1}^{n}b_{n,j}x_{j}- \beta_{n}^x x_2- \frac{1}{2}(\beta^x_{n-1}-\alpha^x_{n-2})x_3 -\frac{1}{4}(\beta^x_{n-2}-\alpha^x_{n-3})x_4- \frac{2^{n-5}+1}{2^{n-5}} \alpha^x_2   x_{n-1}-\alpha_1^xx_{n}  \Bigl)$

 Here  $b_i^x$ $n-2\le i \le n$ are defined arbitrary:

 so  if  $\alpha_i^x$ $n-2\le i \le n$ are defined arbitrary, so we determine $\beta_{n-2}^x, \beta_{n-1}^x , \quad and \quad \beta_{n}^x$

$\beta_{n-2}^x=\frac{1}{x_2}\Bigl( \sum\limits_{j=1}^{2}b_{n-2,j}x_j+ \sum\limits_{j=n-3}^{n-2}b_{n-2,j}x_{j}-\alpha_{n-2}^x x_1-\frac{2^{n-7}+1}{2^{n-7}} \alpha^x_2   x_{n-4}-\alpha_1^xx_{n-2}   \Bigl)$

$\beta_{n-1}^x=\frac{1}{x_2} \Bigl(\sum\limits_{j=1}^{3}b_{n-1,j}x_j+ \sum\limits_{j=n-2}^{n-1}b_{n-1,j}x_{j}- \alpha_{n-1}^x x_1- \frac{1}{2}(\beta^x_{n-2}-\alpha^x_{n-3})x_3- \frac{2^{n-6}+1}{2^{n-6}} \alpha^x_2   x_{n-2}-\alpha_1^xx_{n-1}  \Bigl)$

$\beta_{n}^x=\frac{1}{x_2} \Bigl(\sum\limits_{j=1}^{4}b_{n,j}x_j+ \sum\limits_{j=n-1}^{n}b_{n,j}x_{j}- \alpha_{n}^x x_1- \frac{1}{2}(\beta^x_{n-1}-\alpha^x_{n-2})x_3 -\frac{1}{4}(\beta^x_{n-2}-\alpha^x_{n-3})x_4- \frac{2^{n-5}+1}{2^{n-5}} \alpha^x_2   x_{n-1}-\alpha_1^xx_{n} \Bigl)$
\end{itemize}
Consequently, the linear system~\eqref{local sysytem m2n} is unconditionally solvable for every element $x \in \mathfrak{m}_2(n)$, thereby ensuring that the linear operator induced by the matrix $A$ indeed constitutes a local $\frac{1}{2}$- derivation. Since $\nabla$ was chosen arbitrarily as a local $\frac{1}{2}$- derivation coinciding pointwise with the  $\frac{1}{2}$- derivation characterized in Theorem~\eqref{THM: halfder m2n} it follows that the matrix representation of any local $\frac{1}{2}$- derivation on the algebra $\mathfrak{m}_2(n)$ inherently assumes a lower triangular form.

\end{proof}

\subsection{Local \texorpdfstring{$\frac{1}{2}$}{1/2}-derivation of the algebra \texorpdfstring{$W^+(n)$}{w(n)}}\label{subsec11}

\begin{Th}\label{thm:local-W}
A linear map $\nabla$ on the Lie algebra $W^+(n)$ ($n \ge 12$) is a local $\frac{1}{2}$- derivation if and only if its matrix $M_{\nabla}$ with respect to the basis $\{e_1, e_2, \dots, e_n\}$ has a lower triangular form:
\begin{equation}\label{eq:local-matrix-W-form}
M_{\nabla} = \begin{pmatrix}
b_{1,1} & 0 & 0 & 0 & \cdots & 0 & 0 \\
b_{2,1} & b_{2,2} & 0 & 0 & \cdots & 0 & 0 \\
b_{3,1} & b_{3,2} & b_{3,3} & 0 & \cdots & 0 & 0 \\
b_{4,1} & b_{4,2} & b_{4,3} & b_{4,4} & \cdots & 0 & 0 \\
\vdots & \vdots & \vdots & \vdots & \ddots & \vdots & 0 \\
b_{n-1,1} & b_{n-1,2} & b_{n-1,3} & b_{n-1,4} & \cdots & b_{n-1,n-1} & 0 \\
b_{n,1} & b_{n,2} & b_{n,3} & b_{n,4} & \cdots & b_{n,n-1} & b_{n,n}
\end{pmatrix},
\end{equation}
where $b_{i,j} \in \mathbb{C}$.
\end{Th}
\begin{proof}
By Theorem~\ref{THM: halfder w+n}  the $\frac{1}{2}$- derivation $\mathcal{D}_x$ has the following matrix form:

\[
A_x
=
\begin{pmatrix}\label{matrix form Ax W+n}
\alpha^x_1 & 0 & 0 & 0 & \cdots & 0 & 0 \\
\alpha^x_2 & \alpha^x_1 & 0 & 0 & \cdots & 0 & 0 \\
\alpha^x_3 & \alpha^x_2 & \alpha^x_1 & 0 & \cdots & 0 & 0 \\
\alpha^x_4 & \alpha^x_3 & \alpha^x_2 & \alpha^x_1 & \cdots & 0 & 0 \\
\alpha^x_5 & \alpha^x_4 & \alpha^x_3  & \alpha^x_2 & \cdots & 0 & 0 \\
\vdots & \vdots & \vdots & \vdots & \ddots & \vdots & \vdots \\
\alpha^x_{n-1} & \beta^x_{n-1} & \frac{1}{2} \Bigl(   (n-3)\beta_{n-2} -(n-7)\alpha_{n-3} \Bigl) & a^x_{n-4} & \cdots & \alpha^x_1 & 0 \\
\alpha^x_n & \beta^x_n & \frac{1}{2} \Bigl(   (n-2)\beta_{n-1} -(n-6)\alpha_{n-2} \Bigl) & \frac{1}{8} \Bigl(   (n^2-5n+6)\beta_{n-2} -(n^2-5n-10)\alpha_{n-3} \Bigl) & \cdots & \alpha^x_2 & \alpha^x_1
\end{pmatrix}
\]

In analogy with the proofs of the aforementioned theorems, we construct the following system of governing equations:

\begin{equation}\label{local sysytem Wn}
    \begin{cases}
    b_{1,1}x_1=\alpha_1^x x_1 \\
    \sum\limits_{j=1}^{i}b_{i,j}x_j = \sum\limits_{j=1}^i \alpha^x_{i+1-j}x_j, \quad 4 \le i\le n-3, \\
      \sum\limits_{j=1}^{n-2}b_{n-2,j}x_j=\alpha_{n-2}^x x_1+\beta_{n-2}^x x_2+ \sum\limits_{j=3}^{n-2} \alpha^x_{n-1-j}x_j   \\
     \sum\limits_{j=1}^{n-1}b_{n-1,j}x_j=\alpha_{n-1}^x x_1+\beta_{n-1}^x x_2+\frac{1}{2} \Bigl(   (n-3)\beta^x_{n-2} -(n-7)\alpha^x_{n-3} \Bigl) x_3+ \sum\limits_{j=4}^{n-1} \alpha^x_{n-j}x_j ,  \\           \sum\limits_{j=1}^{n}b_{n,j}x_j=\alpha_{n}^x x_1+\beta_{n}^x x_2+ \frac{1}{2} \Bigl(   (n-2)\beta_{n-1} -(n-6)\alpha_{n-2} \Bigl) x_3 \\ \qquad + \frac{1}{8} \Bigl(   (n^2-5n+6)\beta_{n-2} -(n^2-5n-10)\alpha_{n-3} \Bigl)x_4+  \sum\limits_{j=5}^{n} \alpha^x_{n-j+1}x_j ,  \\
    \end{cases}
\end{equation}

in terms of the variables \(\alpha_{i}^{x}\) $(1 \leq i \leq n)$ and  \(\beta_{i}^{x}\) $(n-2 \leq i \leq n)$.

Now, we prove that the linear operator, defined by the matrix A, is a local $\frac{1}{2}$- derivation. 

\begin{itemize}
\item \textbf{ $x_1 \neq 0$.}
We select the leading scaling parameter of the $\frac{1}{2}$- derivation to be $\alpha_1^x = b_{1,1} \neq 0$ 

$\alpha_2^x=b_{2,1}+\frac{x_2}{x_1}\bigl( b_{2,2}-b_{1,1}\bigl)$

$\alpha_3^x=\frac{1}{x_1}\bigl( \sum\limits_{j=1}^{3}b_{3,j}x_j -  \Bigl( b_{2,1}+\frac{x_2}{x_1}\bigl( b_{2,2}-b_{1,1}\bigl) \Bigl) x_2 -b_{1,1} x_3  \bigl)      $

$\cdots \cdots \cdots \cdots$

$\alpha_i^x=b_{i,1}+\frac{1}{x_1} \Bigl( \sum\limits_{j=2}^i(b_{i,j}-a^x_{i-j+1})\Bigl); $ $ where \quad 4 \le i \le n-3$

$\alpha_{n-2}^x=\frac{1}{x_1}\Bigl( \sum\limits_{j=1}^{n-2}b_{n-2,j}x_j-\beta_{n-2}^x x_2- \sum\limits_{j=3}^{n-2} \alpha^x_{n-1-j}x_j      \Bigl)$

$\alpha_{n-1}^x=\frac{1}{x_1}\Bigl( \sum\limits_{j=1}^{n-1}b_{n-1,j}x_j-\beta_{n-1}^x x_2-\frac{1}{2} \Bigl(   (n-3)\beta^x_{n-2} -(n-7)\alpha^x_{n-3} \Bigl) x_3- \sum\limits_{j=4}^{n-1} \alpha^x_{n-j}x_j   \Bigl)$

$\alpha_n^x=\frac{1}{x_1}\Bigl(\sum\limits_{j=1}^{n}b_{n,j}x_j-\beta_{n}^x x_2- \frac{1}{2} \Bigl(   (n-2)\beta_{n-1} -(n-6)\alpha_{n-2} \Bigl) x_3 \\ \qquad - \frac{1}{8} \Bigl(   (n^2-5n+6)\beta_{n-2} -(n^2-5n-10)\alpha_{n-3} \Bigl)x_4-  \sum\limits_{j=5}^{n} \alpha^x_{n-j+1}x_j  \Bigl)$

 Here  $\beta_i^x$ $n-2\le i \le n$ are defined arbitrary:

so  if  $\alpha_i^x$ $n-2\le i \le n$ are defined arbitrary, so we determine $\beta_{n-2}^x, \beta_{n-1}^x , \quad and \quad \beta_{n}^x$

$\beta_{n-2}^x=\frac{1}{x_2}\Bigl(  \sum\limits_{j=1}^{n-2}b_{n-2,j}x_j-\alpha_{n-2}^x x_1- \sum\limits_{j=3}^{n-2} \alpha^x_{n-1-j}x_j   \Bigl)$

$\beta_{n-1}^x=\frac{1}{x_2} \Bigl(\sum\limits_{j=1}^{n-1}b_{n-1,j}x_j-\alpha_{n-1}^x x_1-\frac{1}{2} \Bigl(   (n-3)\beta^x_{n-2} -(n-7)\alpha^x_{n-3} \Bigl) x_3- \sum\limits_{j=4}^{n-1} \alpha^x_{n-j}x_j  \Bigl)$

$\beta_{n}^x=\frac{1}{x_2} \Bigl(\sum\limits_{j=1}^{n}b_{n,j}x_j-\alpha_{n}^x x_1- \frac{1}{2} \Bigl(   (n-2)\beta_{n-1} -(n-6)\alpha_{n-2} \Bigl) x_3 \\ \qquad - \frac{1}{8} \Bigl(   (n^2-5n+6)\beta_{n-2} -(n^2-5n-10)\alpha_{n-3} \Bigl)x_4-  \sum\limits_{j=5}^{n} \alpha^x_{n-j+1}x_j \Bigl)$
\end{itemize}

It follows that the system of linear equations~\ref{local sysytem Wn} admits a consistent solution for each $x \in W^+(n)$, which implies that the linear map determined by $A$ acts as a local $\frac{1}{2}$- derivation. Given the arbitrary choice of $\nabla$ and its pointwise evaluation via the $\frac{1}{2}$- derivation described in Theorem~\ref{THM: halfder w+n}, we conclude that the matrix representation of every local $\frac{1}{2}$- derivation of $W^+(n)$ is necessarily lower triangular.

\end{proof}

\subsection{Local \texorpdfstring{$\frac{1}{2}$}{1/2}-  derivation of the algebra \texorpdfstring{$m_{0,s}(2k+s)$}{m(0s)}}\label{subsec12}

\begin{Th}\label{thm:local-m01n}
A linear map $\nabla$ on the Lie algebra $m_{0,s}(2k+s)$  is a local $\frac{1}{2}$- derivation if and only if its matrix $M_{\nabla}$ with respect to the basis $\{e_1, e_2, \dots, e_{2k+s}\}$   (in here $n=2k+s, \quad where \quad 1\le s \le 3$ ) has the following lower triangular form:
\begin{equation}\label{eq:local-matrix-m0sn-form}
M_{\nabla} = \begin{pmatrix}
b_{1,1} & 0 & 0 & 0 & \cdots & 0 & 0 \\
b_{2,1} & b_{2,2} & 0 & 0 & \cdots & 0 & 0 \\
b_{3,1} & b_{3,2} & b_{3,3} & 0 & \cdots & 0 & 0 \\
b_{4,1} & b_{4,2} & b_{4,3} & b_{4,4} & \cdots & 0 & 0 \\
\vdots & \vdots & \vdots & \vdots & \ddots & \vdots & 0 \\
b_{2k+2,1} & b_{2k+2,2} & b_{2k+2,3} & b_{2k+2,4} & \cdots & b_{2k+2,2k+2} & 0 \\
b_{2k+3,1} & b_{2k+3,2} & b_{2k+3,3} & b_{2k+3,4} & \cdots & b_{2k+3,2k+2} & b_{2k+3,2k+3}
\end{pmatrix},
\end{equation}
where $b_{i,j} \in \mathbb{C}$ .
\end{Th}

\begin{proof}
The proof for the general case proceeds by analogy with Theorem~\eqref{A_x form matrix m0n}, since the underlying $\frac{1}{2}$-derivations differ exclusively in the component corresponding to the basis vector $e_{2k+1}, e_{2k+2}, e_{2k+3}$, where we now have $b_2 = a_1$. 

$$\nabla(x)=\mathcal{D}_x(x)$$

By Theorem~\ref{THM: halfder m01n}  the $\frac{1}{2}$- derivation $\mathcal{D}_x$ has the following matrix form:
\[
A_x = \begin{pmatrix}\label{local matrix m01n}
\alpha_1^x & 0 & 0 & 0 & \cdots & 0  & 0 \\
\alpha_2^x & \alpha_1^x & 0 & 0 & \cdots & 0 & 0 \\
\alpha_3^x & \beta_3^x & \alpha_1^x & 0 & \cdots & 0 & 0  \\
\alpha_4^x & \beta_4^x & \frac{\beta_3^x}{2}  & \alpha_1^x & \cdots & 0 & 0  \\
\vdots & \vdots & \vdots & \vdots & \ddots & \vdots & 0  \\
\alpha_{2k+1}^x & \beta_{2k+1}^x & \frac{1}{2} (\beta_{2k}^x+\alpha^x_{2k-1}) & \frac{1}{4} (\beta_{2k-1}^x-2\alpha_{2k-2}^x) & \cdots & 0 & 0 \\ 
\alpha_{2k+2}^x & \beta_{2k+2}^x & \frac{1}{2} (\beta_{2k+1}^x+\frac{n-5}{2}\alpha^x_{2k}) & \frac{1}{4} (\beta_{2k}^x-(n-7)\alpha_{2k-1}^x) & \cdots & \alpha_1^x & 0  \\ 
\alpha_{2k+3}^x & \beta_{2k+3}^x & \frac{1}{2}\beta_{2k+2}^x & \frac{1}{2}\beta_{2k+1}+(2-2k)a_{2k}) & \cdots & 0& \alpha_1^x \\
\end{pmatrix}.
\]

Depending on the parameter $s$, the structure of the underlying algebra yields the following case-by-case characterization:
For $s = 1$, the algebra coincides with $\mathfrak{m}_{0,1}(2k+1)$. For $s = 2$, it reduces to $\mathfrak{m}_{0,2}(2k+2)$, under which the parameter constraints $\beta_3 = \alpha_2 = 0$ hold valid. In the case where $s = 3$, the algebra is identified as $\mathfrak{m}_{0,3}(2k+3)$, which requires the simultaneous vanishings $\beta_3 = \alpha_2 = 0$ and $\beta_4 = \alpha_3 = 0$. Consequently, the zero coefficients of the associated matrix $A = (b_{i,j})$ are determined accordingly under these respective constraints.
 Consequently, extending the relation \(A\bar{x} = A_x\bar{x}\) to a generic element \(x = \sum\limits_{i=1}^n x_i e_i\) yields the subsequent system of algebraic equations:

 \begin{equation}\label{local system m0s}
     \begin{cases}
         b_{1,1}x_1=\alpha_1x_1 \\
         b_{2,1}x_1+b_{2,2}x_2=\alpha^x_2x_1+\alpha^x_1x_2 \\
         \sum\limits_{j=1}^{i}b_{i,j}x_j=\alpha_i^x x_1+\beta_i^x x_2+\sum\limits_{j=3}^{i-1}\frac{1}{2^{j-2}} \beta_{i-j+2}^xx_j +\frac{(2^{i-2}-1)\alpha_1^x+\beta_2^x}{2^{i-2}}x_i, 3 \le i \le 2k \\
         \sum\limits_{j=1}^{2k+1}b_{{2k+1},j}x_j=\alpha_{2k+1}^x x_1+\beta_{2k+1}^x x_2+\sum\limits_{j=3}^{2k}\bigl(\frac{1}{2^{j-2}}\beta_{{2k}-j+3}+\frac{1}{2}(-1)^{j+1}\alpha_{{2k}-j+2}  \bigl)x_j+\alpha_1^x x_{2k+1}  \\       
         \sum\limits_{j=1}^{2k+2}b_{{2k+2},j}x_j=\alpha_{2k+2}^x x_1+\beta_{2k+2}^x x_2+\sum\limits_{j=3}^{2k}\bigl(\frac{1}{2^{j-2}}\beta_{{2k}-j+4}+\frac{1}{2}(-1)^{j+1}\bigl(\frac{n+1}{2} \bigl)\alpha_{{2k}-j+3}  \bigl)x_j+\alpha_1^x x_{2k+2}  \\
         \sum\limits_{j=1}^{2k+3}b_{{2k+3},j}x_j=\alpha_{2k+3}^x x_1+\beta_{2k+3}^x x_2+\sum\limits_{j=3}^{2k}\bigl(\frac{1}{2^{j-2}}\beta_{{2k}-j+4}+\frac{1}{2}(-1)^{j}\Bigl( (j-3)\frac{2k}{2}-\frac{(j-2)^2}{2}+\frac{2k+1}{4}  \Bigl) \alpha_{2k-j+4})x_j \\ \qquad +\alpha_1^x x_{2k+3}  \\
     \end{cases}
 \end{equation}

It suffices to evaluate the initial three equations of the system, as the validity of the remaining relations has been established in the proof of Theorem~\eqref{thm:local-m0}.

\begin{itemize}
    \item  If $x_1 \neq 0$, then $a_1^x = b_{1,1}$.
    
$\alpha_{2k+1}^x = \frac{1}{x_1}\Bigl(\sum\limits_{j=1}^{{2k+1}}b_{{2k+1},j}x_j-b_{2k+1}^x x_2-\sum\limits_{j=3}^{2k}\bigl(\frac{1}{2^{j-2}}\beta_{2k-j+3}-\frac{1}{2}(-1)^{j+1}\alpha_{2k-j+2}  \bigl)x_j-\alpha_1^x x_{2k+1}\Bigl),$

$\alpha_{2k+2}^x =\frac{1}{x_1}\Bigl(\sum\limits_{j=1}^{2k+2}b_{{2k+2},j}x_j-\beta_{2k+2}^x x_2 - \sum\limits_{j=3}^{2k}\bigl(\frac{1}{2^{j-2}}\beta_{{2k}-j+4}+\frac{1}{2}(-1)^{j+1}\bigl(\frac{n+1}{2} \bigl)\alpha_{{2k}-j+3}  \bigl)x_j-\alpha_1^x x_{2k+2}   \Bigl) $

$\alpha_{2k+3}^x=\frac{1}{x_1}\Bigl( \sum\limits_{j=1}^{2k+3}b_{{2k+3},j}x_j-\beta_{2k+3}^x x_2-\sum\limits_{j=3}^{2k}\bigl(\frac{1}{2^{j-2}}\beta_{{2k}-j+4}+\frac{1}{2}(-1)^{j}\Bigl( (j-3)\frac{2k}{2}-\frac{(j-2)^2}{2}+\frac{2k+1}{4}  \Bigl) \alpha_{2k-j+4})x_j \\  -\alpha_1^x x_{2k+3}   \Bigl)$

$\beta_k^x $ is arbitrary
\item  if $x_1=0$ and $x_2 \neq 0$, then $b_{2,2}=(a_1^x)$ 

$b_{2k+1}^x =\frac{1}{x_2}\Bigl( \sum\limits_{j=1}^{{2k+1}}b_{{2k+1},j}x_j-\sum_{j=3}^{2k}\bigl(\frac{1}{2^{j-2}}\beta_{2k-j+3}-\frac{1}{2}(-1)^{j+1}\alpha_{2k-j+2}  \bigl)x_j-\alpha_1^x x_{2k+1} \Bigl),$

$b_{2k+2}^x =\frac{1}{x_2}\Bigl( \sum\limits_{j=1}^{2k+2}b_{{2k+2},j}x_j-\sum\limits_{j=3}^{2k}\bigl(\frac{1}{2^{j-2}}\beta_{{2k}-j+4}+\frac{1}{2}(-1)^{j+1}\bigl(\frac{n+1}{2} \bigl)\alpha_{{2k}-j+3}  \bigl)x_j-\alpha_1^x x_{2k+2} \Bigl)$

$b_{2k+3}^x =\frac{1}{x_2}\Bigl(\sum\limits_{j=1}^{2k+3}b_{{2k+3},j}x_j-\sum\limits_{j=3}^{2k}\bigl(\frac{1}{2^{j-2}}\beta_{{2k}-j+4}+\frac{1}{2}(-1)^{j}\Bigl( (j-3)\frac{2k}{2}-\frac{(j-2)^2}{2}+\frac{2k+1}{4}  \Bigl) \alpha_{2k-j+4})x_j  -\alpha_1^x x_{2k+3}   \Bigl)$

give solution of the system of equation~\ref{local system m0s}. Here  $a_1^x$ is arbitrary value.
\end{itemize}
\end{proof}

\subsection{Local \texorpdfstring{$\frac{1}{2}$}{1/2} -derivation of the algebra \texorpdfstring{$g_{k,\alpha}$}{g(kn)}}\label{subsec8}

\begin{Th}\label{thm:local-gk}
A linear map $\nabla$ on the one-parameter family of Lie algebras $\mathfrak{g}_{k,\alpha}$ ($k \in \{7,8,9,10,11\}$) is a local $\frac{1}{2}$- derivation if and only if its matrix representation $M_{\nabla} = (b_{m,j})$ with respect to the homogeneous basis $\{e_1, e_2, \dots, e_k\}$ is lower triangular.
\begin{equation}\label{eq:local-matrix-gk-form}
M_{\nabla} = \begin{pmatrix}
b_{1,1} & 0 & 0 & 0 & \cdots & 0 & 0 \\
0 & b_{2,2} & 0 & 0 & \cdots & 0 & 0 \\
b_{3,1} & b_{3,2} & b_{3,3} & 0 & \cdots & 0 & 0 \\
b_{4,1} & b_{4,2} & b_{4,3} & b_{4,4} & \cdots & 0 & 0 \\
\vdots & \vdots & \vdots & \vdots & \ddots & \vdots & 0 \\
b_{n-1,1} & b_{n-1,2} & b_{n-1,3} & b_{n-1,4} & \cdots & b_{n-1,n-1} & 0 \\
b_{n,1} & b_{n,2} & b_{n,3} & b_{n,4} & \cdots & b_{n,n-1} & b_{n,n}
\end{pmatrix},
\end{equation}
where $b_{m,j} \in \mathbb{C}$ .
\end{Th}

\section{2-Local \texorpdfstring{$\frac{1}{2}$}{1/2} -derivation of a totally graded filiform Lie algebras}\label{sec5}

The present section is devoted to the description of $2$-local $\frac{1}{2}$- derivation of totally graded filiform Lie algebras.

Let $\mathfrak{A}$ be an algebra. A (not necessarily linear) map $\Delta : \mathfrak{A} \to \mathfrak{A}$ is called a $2$-local $\frac{1}{2}$- derivation if, for any elements $x, y \in \mathfrak{A}$, there exists an $\frac{1}{2}$- derivation $\mathcal{D}_{x,y} : \mathfrak{A} \to \mathfrak{A}$ such that
\[
\Delta(x) = \mathcal{D}_{x,y}(x), \qquad \Delta(y) = \mathcal{D}_{x,y}(y).
\]

\begin{Th}
The Lie algebra $\mathfrak{A}$ in Tables~\eqref{tab:filiformn},~\eqref{tab:filiform}  has 2-local $\frac{1}{2}$- derivation which are not $\frac{1}{2}$- derivation.
\end{Th}

\begin{proof}
Let us define a homogeneous non-additive function $f$ on $\mathbb{C}^2$ as follows:
\begin{equation*}
f(x_1, x_2) = \begin{cases}
\dfrac{x_1^3}{x_1^2 + x_2^2}, & \text{if } x_1^2 + x_2^2 \neq 0, \\
0, & \text{if } x_1^2 + x_2^2 = 0,
\end{cases}
\end{equation*}
where $(x_1, x_2) \in \mathbb{C}^2$. Consider the map $\Delta : \mathfrak{A} \to \mathfrak{A}$ defined in the following way:
\begin{equation*}
\Delta(x) = x + f(x_1, x_2)e_n, \quad \text{where } x = \sum\limits_{i=1}^n x_i e_i \in \mathfrak{A},
\end{equation*}
and $\{e_1, e_2, \dots, e_n\}$ is the basis of $\mathfrak{A}$. Since $f$ is not additive, we have that $\Delta$ is not an $\frac{1}{2}$- derivation of $\mathfrak{A}$.

Let us show that $\Delta$ is a 2-local $\frac{1}{2}$- derivation. Let $v, w$ be arbitrary elements of $\mathfrak{A}$. We take a linear map $\Phi : \mathfrak{A} \to \mathfrak{A}$ with a matrix of the following form with respect to the basis $\{e_1, e_2, \dots, e_n\}$:
\begin{equation*}
\Phi = \begin{pmatrix}
1 & 0 & 0 & 0 & \cdots & 0 & 0 \\
0 & 1 & 0 & 0 & \cdots & 0 & 0 \\
0 & 0 & 1 & 0 & \cdots & 0 & 0 \\
0 & 0 & 0 & 1 & \cdots & 0 & 0 \\
\vdots & \vdots & \vdots & \vdots & \ddots & \vdots & \vdots \\
0 & 0 & 0 & 0 & \cdots & 1 & 0 \\
a & b & 0 & 0 & \cdots & 0 & 1
\end{pmatrix}.
\end{equation*}
By the definition of the $\frac{1}{2}$- derivation matrix $A_{\mathfrak{A}}$ , $\Phi$ is a  $\frac{1}{2}$- derivation of $\mathfrak{A}$. We prove that there exist parameters $a$ and $b$ such that:
\begin{equation*}
\Delta(v) = \Phi(v), \quad \Delta(w) = \Phi(w).
\end{equation*}
Indeed, the last two vector equations give us the following system of algebraic equations with respect to the variables $a$ and $b$:
\begin{equation}\label{eq:system-A}
\begin{cases}
a v_1 + b v_2 = f(v_1, v_2) \\
a w_1 + b w_2 = f(w_1, w_2)
\end{cases}
\end{equation}
Clearly, if $v_1 w_2 - v_2 w_1 = 0$, then the vectors $(v_1, v_2)$ and $(w_1, w_2)$ are proportional. Since the right-hand side of the system \eqref{eq:system-A} is structurally homogeneous of degree 1 (i.e., $f(\lambda z_1, \lambda z_2) = \lambda f(z_1, z_2)$), the equations become scalar multiples of each other, and the system \eqref{eq:system-A} has infinitely many solutions.

Else, if $v_1 w_2 - v_2 w_1 \neq 0$, the determinant of the system matrix is non-zero, and the system \eqref{eq:system-A} has a unique solution for $a$ and $b$.

Thus, for any $v, w$ in $\mathfrak{A}$, there exists a  linear $\frac{1}{2}$- derivation $\Phi$ of $\mathfrak{A}$ such that $\Delta(v) = \Phi(v)$ and $\Delta(w) = \Phi(w)$, i.e., $\Delta$ is a  2-local $\frac{1}{2}$- derivation on $\mathfrak{A}$.
\end{proof}

\end{document}